\documentclass[DIV=12,abstract=true]{scrartcl}
\usepackage{amsmath,amssymb,amsthm,enumitem,mathtools,mathrsfs,csquotes,calc}
\usepackage{tikz}
\usetikzlibrary{positioning}
\usetikzlibrary{patterns}
\usepackage[sortcites=true,giveninits=true,maxbibnames=10,block=space]{biblatex}
\AtEveryBibitem{\clearfield{issn}}
\AtEveryCitekey{\clearfield{issn}}
\usepackage{tikz-cd}
\title{On checking $\mathbf{L}^{\boldsymbol{p}}$-admissibility\\
        for parabolic control systems}
\date{April 9, 2024}
\author{\stepcounter{footnote}
        Philip~Preußler\thanks{Corresponding author} \thanks{Department of Applied Mathematics, 
                University of Twente, 
                P.O.~Box 217, 7500~AE Enschede, 
                The Netherlands, 
                \texttt{p.n.preusler@utwente.nl} }
        \and Felix~L.~Schwenninger\thanks{Department of Applied Mathematics,
                University of Twente, 
                P.O.~Box 217, 
                7500~AE Enschede, 
                The Netherlands, 
                \texttt{f.l.schwenninger@utwente.nl}}
       }
\setkomafont{disposition}{\normalcolor\bfseries\rmfamily}

\renewrobustcmd*{\bibinitdelim}{\,}
\theoremstyle{definition}
\newtheorem{definition}{Definition}[section]
\newtheorem{example}[definition]{Example}
\newtheorem{remark}[definition]{Remark}

\theoremstyle{plain}
\newtheorem{theorem}[definition]{Theorem}
\newtheorem{lemma}[definition]{Lemma}
\newtheorem{proposition}[definition]{Proposition}
\newcommand{\C}{\mathbb{C}}
\newcommand{\R}{\mathbb{R}}
\newcommand{\N}{\mathbb{N}}
\newcommand{\Z}{\mathbb{Z}}
\newcommand{\Ll}{\mathrm{L}}
\newcommand{\LL}{\mathcal{L}}
\newcommand{\LLL}{\mathscr{L}}
\newcommand{\Hh}{\mathrm{H}}
\newcommand{\dd}{\,\mathrm{d}}
\newcommand{\ddd}{\mathrm{d}}
\DeclarePairedDelimiter{\abs}{\lvert}{\rvert}
\DeclarePairedDelimiter{\norm}{\lVert}{\rVert}
\DeclarePairedDelimiter{\nnnorm}{\lvert\kern-0.25ex\lvert\kern-0.25ex\lvert}{\rvert\kern-0.25ex\rvert\kern-0.25ex\rvert}
\DeclarePairedDelimiter{\floor}{\lfloor}{\rfloor}
\DeclarePairedDelimiterX{\spr}[2]{\langle}{\rangle}{#1,#2}
\newcommand{\I}{\mathrm{i}}
\newcommand{\e}{\mathrm{e}}
\newcommand{\fourier}{\mathcal{F}}
\newcommand{\ball}[3]{\mathrm{B}_{#1}^{#2}(#3)}
\newcommand{\hausdorff}{\mathscr{H}}
\AtBeginDocument{
	\renewcommand{\Re}{\operatorname{Re}}
	\renewcommand{\Im}{\operatorname{Im}}
    \renewcommand{\phi}{\varphi}
    \renewcommand{\epsilon}{\varepsilon}
}
\DeclareMathOperator{\esssup}{ess\,sup}
\DeclareMathOperator{\csch}{csch}
\setkomafont{footnotereference}{\normalfont} 
\begin{document}
\maketitle
\vspace{-2.25em}
\begin{abstract}
\noindent{In this note we discuss the difficulty of verifying $\Ll^p$-admissibility 
for $p\neq 2$---that even manifests in the presence of a self-adjoint semigroup generator on a Hilbert space---and
survey tests for $\Ll^p$-admissibility of given control
operators. These tests are obtained by virtue of either mapping properties of boundary
trace operators, yielding a characterization of admissibility via abstract interpolation spaces; or through Laplace--Carleson embeddings, slightly extending results from
Jacob, Partington and Pott \cite{jacob_applications_2014}
to a class of systems which are not necessarily diagonal with respect to sequence spaces. Special focus is laid on illustrating the theory by means of examples based on the heat equation on 
various domains.}
\vspace{0.7em}
\end{abstract}
\begin{description}[font=\rmfamily\bfseries,style=sameline,leftmargin=\widthof{\textbf{Keywords: }}]
\item[Keywords:] admissible operator, infinite-dimensional system, Laplace--Carleson embedding, Weiss conjecture.
\item[Mathematics Subject Classification (2020):] 93B28, 93C05, 93C25, 47N70.
\end{description}
\section{Introduction}
\subsection{State space systems and admissibility}
	We consider abstract linear systems of the form 
		\begin{equation}
			\dot{x}(t) = Ax(t) + Bu(t), \qquad t>0, \label{eq:abstractsystem}
		\end{equation}
	on an infinite-dimensional Banach space $X$ which we call the \emph{state space}.
	Here we require that the operator $A \colon X \supset \mathcal{D}(A) \to X$ generates a strongly continuous semigroup
	of linear operators, or $\mathrm{C}_0$-semigroup, on $X$, denoted by $T = (T(t))_{t\geq 0}$. 
    The \emph{state trajectory}, defined for $t\geq 0$, is denoted by $x$ and the \emph{input function} \emph{\textup{(}\emph{or} control function\textup{)}}, also defined for $t \geq 0$, is denoted by $u$. We require $u$ to be $U$-valued, where $U$ is a Banach space that we call the \emph{input space}. Moreover, inputs $u$ enter the system through the \emph{control operator} $B$ only.
 
	In the field of infinite-dimensional systems theory there
    is a wide array of literature dealing with bounded control operators; that is to say, systems with $B\in\LL(U,X)$. However, many systems arising from controlled partial differential equations lead to
	control operators that fail to be bounded as maps from $U$ to $X$; for example in the
    case of control acting on the boundary. This leads us to the study of
	unbounded control operators $B$. 
 
    Here $B$ is called an \emph{unbounded} control
	operator if $B$ is a linear operator which is bounded as a map from  $U$ to
 $X_{-1}$, but not in $\LL(U,X)$. The extrapolation space $X_{-1} \supset X$ is defined as the completion 
	of $X$ with the weaker norm 
	\[\norm{z}_{X_{-1}} \coloneq \norm{(sI - A)^{-1} z}_X\]
	for some complex number $s$ in the resolvent set $\rho(A) $. 
 
    This allows us to
    consider \eqref{eq:abstractsystem} in the ambient space $X_{-1}$.
	To that end, consider a fixed initial value $x(0) \in X$ and an unbounded control operator
    $B$. The mild solution of the controlled system \eqref{eq:abstractsystem} 
    is given by
		\begin{equation}
			x(t) = T(t) x(0) + \int_0^t T(t-s) Bu(s) \dd s,\qquad t \geq 0, \label{eq:mildsol}
		\end{equation}
	and takes values which a priori only lie in $X_{-1}$. As we are interested in the
	case with continuous state trajectories taking values in $X$\kern-.1em{}---such as in the case
    of bounded $B$---we want to find conditions on $B$ such that \eqref{eq:mildsol} 
	lies in $X$ for all $t\geq 0$, bringing us to the property of admissibility, \cite{weiss_admissibility_1989,jacobAdmissibilityControlObservation2004,weissAdmissibleObservationOperators1989}.
	\begin{definition}[$\Ll^p$-admissible control operators]
		For $1 \leq p \leq\infty$,  {the control operator $B \in \LL(U,X_{-1})$} is called 
    		\emph{finite-time $\Ll^p$-admissible for $T$ \kern-.15em{}\textup{(}\kern-.03em{}or for $A$\textup{)}}
		if the convolution type Bochner integral \[\int_0^t T(t-s) Bu(s) \dd s\]  {is 
        an element of $X$} for any $t \geq 0$ 
        and  {there exists $C>0$} such that 
			\begin{equation*}
				\norm*{\int_0^t T(t-s) Bu(s) \dd s}_{X} \leq C \norm{u}_{\Ll^p([0,t], U)}
			\end{equation*}
		holds for all input functions $u \in \Ll^p([0,t], U)$.
  
		The operator $B$ is called \emph{infinite-time $\Ll^p$-admissible for $T$ \kern-.15em{}\textup{(}\kern-.03em{}or for $A$\textup{)}} if
		the constant $C > 0$ from above can be chosen independently of $t$.
	\end{definition}
    \begin{remark}
    Some remarks on the definition of $\Ll^p$\kern+.06em{}-admissibility of control operators:    
    \begin{enumerate}[leftmargin=3\parindent,rightmargin=3\parindent] 
        \item If the context is clear, one can also 
        drop the reference to the semigroup $T$ or the generator $A$ when referring to admissibility. Moreover, one can also drop the letter $\mathrm{L}$ and speak of $p$\kern+.06em{}-admissibility of control operators. Lastly, when the distinction between finite-time and infinite-time admissibility is not made, one generally speaks about infinite-time admissibility.
        \item We also note that---due to translation
        invariance of the set of input functions---one can replace $T(t-s)$ by
        $T(s)$ in the integrals above. Moreover, infinite-time $p$\kern+.06em{}-admissibility is equivalent to  {the existence of $C > 0$ such that}
        \begin{equation*}
				\norm*{\int_0^\infty T(s) Bu(s) \dd s}_{X} \leq C \norm{u}_{\Ll^p([0,\infty) , U)}
			\end{equation*}
		 {holds for all $u \in \Ll^p([0,\infty),U)$}.
        \item The property that $B\in\LL(U,X_{-1})$ is $p$\kern+.06em{}-admissible can be rephrased by saying that the mapping $\Phi_{t}\colon u\mapsto \int_{0}^{t}T(t-s)Bu(s)\dd s$ is bounded from $\Ll^p$ to $X$\kern-.1em{}---rather than only mapping to $X_{-1}$---for some and hence all $t>0$. Dualizing this statement, we arrive at 
        an admissibility notion for \emph{observation operators}
        $C \colon \mathcal{D}(A) \to Y$\!, where $\mathcal{D}(A)$ is equipped with the graph norm of $A$ and the Banach space  $Y$ is called the \emph{output space}.
        
        More precisely, we fix $p\in(1,\infty)$ and denote the Hölder conjugate of $p$ by $p'$. Then, provided that the dual semigroup $T'=(T(t)')_{t\ge0}$ is strongly continuous on the Banach space dual $X'$, $B\colon U\to X_{-1}$ is $p$\kern+.06em{}-admissible if and only if the operator $C=B'\colon \mathcal{D}(A')\to U'$ satisfies 
        \begin{equation*}
        \int_0^t \norm{B'T(t)'x}_Y^{{ p'}} \dd t \leq K \norm{x}_{X'}^{p}, \qquad x \in \mathcal{D}(A')
        \end{equation*}
        for some $K>0$, see e.g.\  {\cite[Theorem~6.9]{weissAdmissibleObservationOperators1989}}.\end{enumerate}
    \end{remark}
    \subsection{Boundary control systems}
	The most natural form of many systems of partial differential equations 
	with control on the boundary of the domain is not the state space form as in \eqref{eq:abstractsystem} above. 
	Rather, in these systems the control enters via the boundary condition;
	for example in the Dirichlet trace sense $x|_{\partial \Omega} = u$. This then leads to the type of
	systems known as \emph{boundary control systems}, see e.g.\ Salamon \cite{salamonInfiniteDimensionalLinear1987}
    and Tucsnak and Weiss \cite{tucsnak_observation_2009}. Boundary control systems take the form
		\begin{equation}
        \left\{
        \begin{aligned}
					\dot x(t) &= \mathfrak{A} x(t) & & \text{on } (0,\infty),\\
			\mathfrak{B} x(t) &= u(t) & &\text{on } (0,\infty),\\
			x(0) &= x_0. & & 
        \end{aligned}
        \right.
        \label{eq:BCS}
		\end{equation}
	Typically, $\mathfrak{A}$ is a differential operator and $\mathfrak{B}$ 
    is a boundary trace operator, such that the system corresponds to
	a (partial) differential equation with the control input acting via the 
	boundary of the spatial domain. To relate this formulation to the state space form, we need some well-posedness assumptions.
	\begin{definition}[{boundary control systems}]
		Given Banach spaces $X$ and $U$ and closed operators 
        \[
            \mathfrak{A} \colon X \supset \mathcal{D}(\mathfrak{A}) \to X,\qquad 
            \mathfrak{B} \colon X \supset \mathcal{D}(\mathfrak{B}) = 
            \mathcal{D}(\mathfrak{A}) \to U,
        \] 
        where we have normed the space $\mathcal{D}(\mathfrak{A})$ with the graph norm. If now the restriction of $\mathfrak A$ to 
        $\ker \mathfrak B$ is the generator of a $\mathrm{C}_0$-semigroup on $X$
		and $\mathfrak{B}$ has a right inverse $B_0 \in \LL(U,\mathcal{D}(\mathfrak{A}))$,
		then we call the pair of operators $(\mathfrak{A},\mathfrak{B})$ and the corresponding equations \eqref{eq:BCS}
		a \emph{\textup{(}linear\textup{)}\,boundary control system}.  
	\end{definition}
	\begin{remark}
		One can translate a boundary control system on $X$ into a state space system on $X$ by
		setting 
        \begin{equation}
        A \coloneq \mathfrak{A}|_{\ker \mathfrak{B}},\qquad B = (\mathfrak{A} - A_{-1})B_0,
        \label{eq:BCStostate}
        \end{equation}
        where $A_{-1} \colon X_{-1} \supset \mathcal{D}(A_{-1}) \to X_{-1}$ is  {the unique isometric extension of $A$ to the space $X_{-1}$ with domain $\mathcal{D}(A_{-1}) = X$, see e.g.\ \cite{staffans_well-posed_2005, tucsnak_observation_2009, engel_one-parameter_2000}.}
		More on the rewriting procedure for boundary control systems into state space form can e.g.\ be found 
		in \cite{schwenninger_input--state_2020} or \cite[Chapter~10]{tucsnak_observation_2009}.
		
        This correspondence allows us
		to transfer concepts formulated in the state space system framework---including admissibility---to the boundary control context. The intended meaning
		here is that the operator $B$ that results from $\mathfrak{B}$ through the rewriting procedure
		is admissible for $A = \mathfrak{A}|_{\ker \mathfrak{B}}$. While the right inverse
		involved in the construction of $B$ is generally not unique, the operator $B$ is uniquely determined. Yet, computing $B$ from the expression given in \eqref{eq:BCStostate} above is usually not feasible. However, on Hilbert spaces, the relation 
		\begin{equation}
			\spr{\mathfrak{A} x}{y} - \spr{x}{A^* y} = \spr{\mathfrak{B}x}{B^* y} ,\label{eq:hilbertcontrolop}
		\end{equation}
		see e.g.\ \cite[Proposition 2.9]{schwenninger_input--state_2020},
		holds for all $x\in \mathcal{D}(\mathfrak{A})$ and $y \in \mathcal{D}(A^*)$. 
        With this identity, we can identify the control
		operator $B$, simplifying the computation in contrast to \eqref{eq:BCStostate}.
	\end{remark}
    \subsection{Current state of $\mathbf{L}^{\boldsymbol{p}}$-admissibility theory}
	When verifying $p$\kern+.06em{}-admissibility, there is a vast difference between 
    the cases $p = 2$ and $p \neq 2$; a result of the missing Hilbert space structure
    in the latter case. This discrepancy is reflected in the amount of results 
    available in the literature on $2$-admissibility compared to  
    $p$\kern+.06em{}-admissibility, see also the survey article \cite{jacobAdmissibilityControlObservation2004} by Jacob and Partington 
    that deals mostly with the case $p = 2$.
    As an example,
    for $p=2$ characterizations of admissibility using Lyapunov 
    equations and Lyapunov inequalities are available, 
    compare also \cite[Chapter~5]{tucsnak_observation_2009}.
    These methods do not generalize to $p\neq 2$
    as they rely on the Hilbert space structure of the state space $X$, of the input space $U$ and of the $\Ll^2$ space in the domain of the input functions; 
    illustrating the added difficulty of checking for $p$\kern+.06em{}-admissibility.
    
    This disparity could lead to the question of practical relevance of the more general concept of 
    $p$\kern+.06em-{}admissibility compared to $2$-admissibility.
    In other words, one could
    ask if there are any common systems that are not $2$-admissible but only
    $p$\kern+.06em{}-admissible for some $p > 2$. Concretely, we are interested in finding examples of systems arising from partial differential equations---on a natural
	state space such as $\Ll^2(\Omega)$ for some domain 
    $\Omega \subset \R^n$---such that the corresponding control operator $B$ 
    is not $2$-admissible.
	Examples of this type exist; for instance the heat equation on the spatial 
	domain $\Omega = [0,1]$ with Dirichlet boundary control and 
    scalar input space $U = \C$. 
	Here the control operator $B$ resulting from rewriting 
	the boundary control system into state space form turns out to be 
	$B = \delta_0'$, which is $p$\kern+.06em{}-admissible for all $p > 4$ and 
	not $p$\kern+.06em{}-admissible for $p \leq 4$, see \cite{lasiecka_control_2000} and also 
    Section \ref{sec:dirichletheateq} in the sequel. Moreover, $p$\kern+.06em{}-admissibility (in a weighted variant\kern-.04em{}) appears as a prerequisite to the abstract Kato method in the article \cite{haakKatoMethodNavier2009} by Haak and Kunstmann to derive results for Navier--Stokes equations.
    
    In the context of abstract linear systems, $p$\kern+.06em{}-admissibility for 
    $p \neq 2$ has been studied in the  {operator-theoretic} framework
    by and since Weiss's \cite{weiss_admissibility_1989,weissAdmissibleObservationOperators1989} seminal works on observation  and control operators. Weiss also justifies the assumption 
    $B \in \LL(U,X_{-1})$ in the definition of an admissible operator by showing that 
    each ``admissible'' operator $B$ (that would be bounded only from $U$ to some space $V$ in which $X$ lies dense) is equivalent to an operator $\hat B$ that is bounded $U \to X_{-1}$ in the sense that they lead to the same abstract linear system. Furthermore, tests for $p$\kern+.06em{}-admissibility are given for some special cases,
    such as $1$-admissibility on reflexive $X$, as Weiss \cite{weiss_admissibility_1989} shows that $B$ is admissible in this 
    situation if and only if $B \in \LL(U,X)$. 
    
    It should be mentioned that the concept of admissibility---with respect to $\Ll^2$---had previously appeared in Salamon \cite{salamonInfiniteDimensionalLinear1987} as hypothesis (S2). Connections between
    state space systems and boundary control systems are also discussed in the aforementioned work by
    Salamon, and we also mention the work by Curtain and Pritchard \cite{curtainInfiniteDimensionalLinear1978} for more on this topic. 

    Among approaches to characterizing admissibility, important examples are
    Carleson measure criteria for diagonal systems, which
    are known as such due to being based on certain discrete
    measures with weights depending on the system in question being Carleson measures.
    A method based on these conditions for checking $2$-admissibility with scalar input space $U = \C$ 
    appears in Ho and Russell \cite{hoAdmissibleInputElements1983}, was extended to
    an equivalent condition by Weiss \cite{weissAdmissibilityInputElements1988} and was further expanded to control operators defined on $\ell^2(\N)$ by Hansen and Weiss \cite{hansenOperatorCarlesonMeasure1991}. Generalizations of these results
    to the context of $\Ll^p$-admissibility for diagonal semigroups on state spaces of  {$\ell^q$-type} using Laplace--Carleson embeddings can be
    found in works by Haak \cite{haakCarlesonMeasureCriterion2010}, by Unteregge \cite{untereggePAdmissibleControlElements2007} by Jacob, Partington and Pott  \cite{jacob_applications_2014} and by Jacob, Partington, Pott, Rydhe and Schwenninger \cite{jacobLaplaceCarlesonEmbeddingsInfinitynorm2023}. In this frame of reference we also
    mention that extensions to normal semigroups have appeared in Weiss \cite{weissPowerfulGeneralizationCarleson1999}, Hansen and Weiss \cite{hansenOperatorCarlesonMeasure1991}, Unteregge \cite{untereggePAdmissibleControlElements2007} and Staffans \cite{staffans_well-posed_2005}.
    
    Another property of note which is used to derive results for admissibility 
    is the Weiss property and its $\Ll^p$-version. 
    For $p\in[1,\infty]$ the \emph{$p$\kern+.06em{}-\kern-.15em{}Weiss property for control operators}
	for the generator $A$ is the statement that for any Banach space $U$ and any $B\in\LL(U,X_{-1})$ we have the equivalence
	\[
	\text{$B$ is infinite-time $p$\kern+.06em{}-admissible }\Leftrightarrow 
	\sup_{\Re \lambda > 0}(\Re \lambda)^{\frac 1p}\norm{R(\lambda,A) B} < \infty;
	\]
    for the analogous formulation for observation operators see
    Section~\ref{sec:applicationsinfdim} in the sequel.
    
    The property
    originates from conjectures formulated by Weiss in 
    \cite{weissTwoConjecturesAdmissibility1991}, which postulated
    the equivalence of a resolvent condition and admissibility. In
    \cite{weissTwoConjecturesAdmissibility1991} the author remarked that the conjecture is not true on
    general state spaces $X$ that are not Hilbert spaces. Later it was shown that the
    conjecture is false even in the Hilbert space setting.
    For counterexamples see Jacob, Partington and Pott \cite{jacobADMISSIBLEWEAKLYADMISSIBLE2002} based on the right shift semigroup; compare with Zwart and Jacob \cite{jacobDisproof2000} and Zwart, Jacob and Staffans \cite{zwartWeakAdmissibilityDoes2003a} with counterexamples where the semigroups are even analytic. However---in special cases---positive results on the Weiss property were
    also found, e.g.\ by Jacob and Partington \cite{jacobWeissConjectureAdmissibility2001} for contraction semigroups with $p = 2$ and finite-dimensional input spaces, and by Le~Merdy \cite{lemerdyWeissConjectureBounded2003} in the case of $2$-admissibility for bounded analytic semigroups satisfying a square function estimate and arbitrary input spaces. 
    
    In conjunction with Le~Merdy's article  \cite{lemerdyWeissConjectureBounded2003} the connection with
    functional calculus methods and square function estimates must also be mentioned, which are used as a powerful tool 
    to generate estimates in the admissibility context. With these methods, Le Merdy showed that if $A$ has dense range and
    generates a bounded analytic semigroup, then admissibility of $(-A)^{1/2}$ in the observation operator sense---which is a square function estimate for $A$---is equivalent to $A$ having the $2$-\kern-.15em{}Weiss property.
    That is to say that $C \colon X_1 \to Y$ is $2$-observation-admissible for $A$ if and only if 
    \[
        \sup_{\Re \lambda > 0} \sqrt{\Re \lambda}\, \norm{C R(\lambda, A)} < \infty.
    \]
    Note that for Hilbert spaces, the Weiss property is specifically satisfied if $A$ is densely defined and generates an analytic contraction semigroup.
    
    Further results in this direction appear in the work \cite{haakLeMerdyAlphaAdmissible2005} by Haak and Le Merdy, where the $\mathrm{H}^\infty$ functional calculus
    is used to show the equivalence of certain square function estimates and weighted $2$-admissibility for analytic semigroups.
    The functional calculus also is an integral
    part of the dissertation \cite{haakKontrolltheorieBanachraeumenUnd2005} by Haak, where extensions of Le~Merdy's results to $\Ll^p$-admissibility are treated with $\Ll^p$ estimates
    in lieu of the square function estimates that appear in the $\Ll^2$ setting. In particular, $p$\kern+.06em{}-admissibility is characterized under the assumption that
    $(-A)^{1/p}$ is $p$\kern+.06em{}-admissible for $A$. The article
    \cite{haakWeightedAdmissibilityWellposedness2007} by Haak and Kunstmann features 
    generalized results in this direction. Specifically, the authors investigate the more general property of weighted $\Ll^p$-admissibility of type $\alpha$.
    Here so\kern+.02em{}-called $\Ll^p_*$ functional calculus estimates are used, generalizing the square function estimates appearing in the condition of 
    admissibility of $(-A)^{1/2}$. With this assumption, the authors prove that the
    Weiss resolvent condition in the $\Ll^p$ version is equivalent to $p$\kern+.06em{}-admissibility of $B$. They also show that $-A$ has $\Ll^p_*$
    estimates if and only if the state space $X$ embeds into the real interpolation space $(\dot{X}_{-1},\dot{X}_1)_{1/2,p}$. This leads to results for concrete 
    example systems if full expressions for these interpolation spaces are
    known, such as the Neumann controlled heat equation on $\Ll^q$ spaces or
    zero smoothness Besov spaces $\mathrm{B}^0_{q,p}$. Furthermore, a test for 
    the resolvent condition in the Weiss property via real interpolation spaces is given.
    
    For more applications of functional calculus methods to
    proofs of admissibility, we refer to Haak \cite{haakWeissConjectureWeak2012}, where
    the interplay of the Weiss condition and admissibility with respect to the weak-type
    space $\Ll^{2,\infty}$ is discussed. 
    Adding to the positive results on the Weiss conjecture, in \cite{hamidDirectApproachWeiss2010} the authors
    Bounit, Driouich and El-Mennaoui characterize the $p$\kern+.06em{}-Weiss
    condition for bounded analytic semigroups  by admissibility
    of $(-A)^{1/p}$ without resorting to 
    the $\mathrm{H}^\infty$ functional calculus.

    Finally we remark that the above list is not exhaustive, as there are even more methods of showing $\Ll^p$-admissibility
    such as semigroup generation criteria. These are methods based on block operator matrices generating $\mathrm{C}_0$-semigroups; see
    Grabowski and Callier \cite{grabowskiAdmissibleObservationOperators1996}, 
    Engel \cite{engelCharacterizationAdmissibleControl1998} and the recent article 
    \cite{hosfeldCharacterizationOrliczAdmissibility2023}
    by Hosfeld, Jacob and Schwenninger, which treats the more general case of admissibility with respect to Orlicz spaces. 
    We further stress that $\infty$\kern+.03em{}-admissibility, which is formally very closely related to the notion of input-to-state stability ({\textsc{iss}}), see e.g.\  the survey article \cite{mironchenkoInputtoStateStabilityInfiniteDimensional2020}, is not the focus
    of the present note. However, as $p$\kern+.06em{}-admissibility for any finite $p$ implies $\infty$\kern+.03em{}-admissibility, the discussed results serve as natural sufficient conditions for {\textsc{iss}}.
    \subsection{Aims}
    The methods of verifying $p$\kern+.06em{}-admissibility---for $p \neq 2$---that we set our focus  {on} can be grouped into
	\begin{enumerate}[leftmargin=3\parindent,rightmargin=3\parindent,label=(\kern+.035em\roman*\kern+.035em{})]
		\item 	sufficient conditions for analytic semigroups using mapping properties of boundary trace operators, see e.g.\ \cite{lasiecka_control_2000,schwenninger_input--state_2020,maragh_admissible_2014},
		\item 	tests arising from Laplace--Carleson 
				embeddings for diagonal semigroups with finite-dimensional input
                spaces, see e.g.\ \cite{jacob_applications_2014,jacobLaplaceCarlesonEmbeddingsInfinitynorm2023,untereggePAdmissibleControlElements2007,haakCarlesonMeasureCriterion2010}, 
		\item 	applying $p$\kern+.06em{}-\kern-.15em{}Weiss property type results in cases where the
				property is valid, see e.g.\ \cite{lemerdyWeissConjectureBounded2003, weissTwoConjecturesAdmissibility1991,haakAdmissibilityUnboundedOperators2006,haakWeightedAdmissibilityWellposedness2007,haakKontrolltheorieBanachraeumenUnd2005,hamidDirectApproachWeiss2010}.
	\end{enumerate}
    Our aim is to provide a comparison---with some extensions---of the aforementioned methods of
    verifying $p$\kern+.06em{}-admissibility while reviewing applicability to infinite-dimensional input spaces and potential problems in applications.
    In this setting, we also comment on the case $p > 2$ posing even more
    difficulties than the case $p \leq 2$.
    
    Elaborating on this, we discuss that (\kern+.045em{}i\kern+.045em{}) even leads to 
	a sort of characterization of $p$\kern+.06em{}-admissibility for values of $p$ strictly greater than
	a certain threshold, which depends on the interpolation space that the 
	control operator $B$ maps into, see also \cite{maragh_admissible_2014}.

	For (\kern+.045em{}ii\kern+.045em{}) we note that $\Ll^p$-admissibility of the input element
    $b \in \LL(\C,X_{-1})$ is 
	equivalent to boundedness of the Laplace transform operator $\LLL$, which is
    to be understood as a mapping 
	\[
        \LLL \colon \Ll^p([0,\infty)) \to \Ll^2(\C_+, \mu).
    \]
    The measure $\mu$ is given by
	the discrete measure $\mu = \sum_k\,\abs{b_k}^2\delta_{-\lambda_k}$
	in the diagonal case, i.e.~in the presence of a Riesz basis 
	of eigenvectors of the generator
	with eigenvalues $\{\lambda_k\}_{k\in\N}$, cf.~results by Jacob, Partington and Pott \cite{jacob_applications_2014}.
	However, the conditions given in their work are not applicable if 
	the generator does not have compact resolvents as the spectrum is then not discrete.
	For instance, if the generator $A$ is a normal operator, the spectral
    theorem yields a spectral measure $E$ (that  {is not necessarily discrete}) associated to $A$, admitting the representation
    \[
        A = \int_{\sigma(A)} z \dd E (z).
    \]
    Using this fact, we are able to generalize the diagonal results.
	The extensions are established using a suitable generalization of the 
    discrete measure $\mu$ related to the measure $\spr{Eb}{b}$; see 
also similar considerations for the special case $p=2$ in \cite{weissPowerfulGeneralizationCarleson1999,staffans_well-posed_2005,untereggePAdmissibleControlElements2007,hansenOperatorCarlesonMeasure1991}. We also discuss a related
    result for systems with semigroups generated by multiplication operators on
    $\Ll^q$ spaces, generalizing the setting where one has a $q$\kern+.06em{}-Riesz basis. This is related to diagonal systems on $\ell^q$ sequence spaces as in 
    \cite{jacob_applications_2014}.
 
    As an outlook, and relating to (\kern+.045em{}iii\kern+.045em{}), we provide an argument based on the $p$\kern+.06em{}-\kern-.15em{}Weiss property 
    which allows the extension of results for scalar input spaces to 
    infinite-dimensional input spaces. In certain special cases where 
    the $p$\kern+.06em{}-\kern-.15em{}Weiss property does hold, such 
    as for negative semidefinite self-adjoint generators when $p \leq 2$,
    this provides a tool to assess $p$\kern+.06em{}-admissibility for the critical value of $p$; that is to say the value of $p$ that defines the threshold between admissibility and non-admissibility.
	\section{Interpolation spaces, extrapolation spaces and admissibility}
	\subsection{Mapping properties and their relation to admissibility}
    \label{sec:interpolation}
    Here and throughout the rest of the paper, we denote the 
    resolvent of an operator $A$ by $R(z,A) \coloneq (zI - A)^{-1}$ and the 
    range of $A$ by the symbol $\mathcal{R}(A)$.
    Moreover, we denote the \emph{growth bound} of the semigroup $T$ by $\omega(T)$ and we recall that $T$ is called \emph{exponentially stable} if $\omega(T) < 0$.
		A $\mathrm{C}_0$-semigroup $T$ is called \emph{analytic} if there exists $\phi \in (0,{\pi}/{2}]$ such
		that $T$ can be extended to an analytic map on the open sector of angle $2\phi$ around the
		positive real line defined by 
		$S_\phi \coloneq \{ z \in \C \setminus \{0\} : \abs{\operatorname{Arg} z} < \phi \}$.
	In the presence of an analytic semigroup $T$ and a generator $A$ such that $0 \in \rho(A)$, we define a scale of interpolation and extrapolation
	spaces $X_{\beta}$ for $\beta \in (-1,1)$. These spaces lie between the spaces $X_1 = 
    \mathcal{D}(A)$, $X_0 \coloneq X$ and $X_{-1}$. Using this scale of spaces allows for the description of finer mapping
	properties of control operators. This is meant in the sense that if $B$ maps into a space $X_{-\beta}$ for some 
	$\beta \in (0,1)$, then we can use this information to make statements about admissibility.
	In defining these spaces, we make use of fractional powers $(-A)^{-\beta} \colon X \to X$ which are defined by the operator-valued integral
	\[
		(-A)^{-\beta} \coloneq \int_\gamma z^{-\beta} R(z,-A) \dd z,
	\]
	where the curve $\gamma$ is the boundary of a sector of the form $S_\phi$ as above that contains $\sigma(-A)$, see also
	\cite[Chapter~II.5]{engel_one-parameter_2000} and \cite{schwenninger_input--state_2020}. 
	The corresponding positive power $(-A)^\beta$ is then defined as the inverse of $(-A)^{-\beta}$, i.e.
    \[
        (-A)^\beta \coloneq
	   ((-A)^{-\beta})^{-1} \colon \mathcal{R} ((-A)^{-\beta}) \to X.
    \]
		Let $A$ generate an analytic semigroup on $X$ with growth bound less than $0$.
		For any $\beta \in (0,1)$ define the space $X_{-\beta}$
		as the completion of $X$ with respect to the norm given on $X$ by $\norm{x}_{-\beta} \coloneq 
		\norm{(-A)^{-\beta} x}_X$.
		Furthermore, we define the space $X_\beta$ as the set $\mathcal{D}((-A)^\beta) = \mathcal{R}((-A)^{-\beta})$
		with the norm $\norm{x}_{\beta} \coloneq \norm{x}_X + \lVert{(-A)^\beta x}\rVert_X$.
		These spaces are also known as \emph{abstract Sobolev spaces}.
        For general analytic semigroups the spaces $X_\beta$ are defined by considering the shifted generator $A-\lambda I$ for sufficiently large $\lambda$, noting that this definition is independent of the choice of $\lambda$ and hence also consistent.
 
	In the following we present a characterization of $p$\kern+.06em{}-admissibility that is derived from the mapping property that
	$B$ be bounded into $X_{-\beta}$ for some $\beta \in (0,1)$, compare also
    \cite{maragh_admissible_2014}. We also record that for $p \in [1,\infty]$ we use the notation $p'$ for the Hölder 
    conjugate (or conjugate exponent) satisfying $1/p + 1/p' = 1$ with the usual convention for the conjugated exponents of $1$ and $\infty$. The following lemma is essentially well-known in the context of admissibility for analytic semigroups. For the sake of the reader we provide the argument.
	\begin{lemma}\label{lem:interpolation}
		Let $A$ be the generator of an analytic semigroup $T$ on 
		$X$. Fix $0 < \beta < 1$. If $B$ is bounded as an operator 
		$U \to X_{-\alpha}$ for all $\alpha > \beta$, then 
		$B$ is finite-time $p$\kern+.06em{}-admissible for all $p > \frac{1}{1-\beta} = (\beta^{-1})'$.
	\end{lemma}
	\begin{proof}
	We fix $0 < \beta < 1$ and let $B \in \LL(U,X_{-1})$.
 Without loss of generality, assume that $T$ is exponentially stable. Assume 
	$B \in \LL(U,X_{-\alpha})$ for an arbitrary $\alpha > \beta$. Then $(-A)^{-\alpha}B \in 
	\LL(U,X)$. Note that for $A$ generating an analytic semigroup we have
	\[
		\norm{(-A)^\alpha T(t)}_{\LL(X)} \leq C_\alpha t^{-\alpha} \e^{t\omega}
	\]
	for $t > 0$ and $\omega > \omega(T)$ and some constant $C_\alpha > 0$, see \cite[Theorem~6.13]{pazySemigroupsLinearOperators1983}. Using these properties,
	we can deduce that
	\begin{align*}
				\norm{T(t)B}_{\LL(U,X)} &= \norm{T(t) (-A)^\alpha (-A)^{-\alpha} B}_{\LL(U,X)}\\
				&\leq \norm{T(t) (-A)^\alpha}_{\LL(X)} \norm{(-A)^{-\alpha} B}_{\LL(U,X)}\\
				&\leq C_\alpha t^{-\alpha} \e^{t\omega} \norm{(-A)^{-\alpha} B}_{\LL(U,X)}\\
				&\leq C_\alpha t^{-\alpha} \e^{t\omega} \norm{B}_{\LL(U,X_{-\alpha})}
	\end{align*}
	holds for $t > 0$ and $\omega > \omega(T)$. Applying this to the  {left-hand} side of the
	defining inequality for $\Ll^p$-admissibility, we get 
	\begin{align*}
				\norm*{\int_0^t T(t-s) Bu(s) \dd s}_X  &\leq \int_0^t \norm{T(t-s) Bu(s)}_X 
				\dd s\\
				&\leq C_\alpha \int_0^t (t-s)^{-\alpha} \e^{(t-s)\omega} 
				\norm{B}_{\LL(U,X_{-\alpha})} \norm{u(s)}_U \dd s.
	\end{align*}
	We can now use Hölder's inequality to further estimate this quantity and find 
	\begin{equation*}
		\int_0^t (t-s)^{-\alpha} \e^{(t-s)\omega} \norm{u(s)}_U \dd s
		\leq \norm{u}_{\Ll^p([0,t],U)} \biggl(\int_0^t ( (t-s)^{-\alpha} \e^{(t-s)\omega} )^{p'}
		\mathrm{d} s\biggr)^{1/p'}\!.
	\end{equation*}
	The  {right-hand} side is finite if and only if $p' < \frac 1 \alpha$. Combining the 
	two above inequalities, we get 
	\begin{equation*}
		\norm*{\int_0^t T(t-s) Bu(s) \dd s}_X \leq K \norm{u}_{\Ll^p([0,t],U)}
	\end{equation*}
	for $p' < \frac 1 \alpha$ and some constant $K>0$. Thus $B$ is $p$\kern+.06em{}-admissible for all $p > \frac{1}{1-\beta}$.
	\end{proof}
	We want to show that the converse of this statement also holds. To do this, we need
	to introduce the scale of Favard spaces, which are another useful class of intermediate spaces
	as there are embeddings that allow us to derive $\Ll^p$ estimates. We give a short summary of the needed results and properties of these spaces.
	\begin{definition}[Favard spaces \cite{engel_one-parameter_2000}]
	Let $T$ be an exponentially stable $\mathrm{C}_0$-semigroup.
 For $0 < \alpha \leq 1$
	we define the \emph{$\alpha$-Favard norm} by
	\begin{equation*}
		\norm{x}_{F_\alpha} \coloneq \sup_{t>0}{t^{-\alpha}\norm{T(t)x - x}_X}
	\end{equation*}
	and the \emph{Favard space $F_\alpha$} as the set of all $x \in X$ with 
	finite $\alpha$-Favard norm.

	To extend the definition to arbitrary real indices $\alpha$, we decompose the 
    number $\alpha \in \R$ into
	an integer part and a fractional part; that is to say we write $\alpha = k + \beta$ (uniquely) with $k\in\Z$ and 
	$0 < \beta \leq 1$ and define $F_\alpha$ to be the Favard space of order $\beta$
	for the extended (resp.~restricted) semigroup $T_k$. Moreover, if $T$ is not exponentially stable, we define $F_\alpha$ as the $\alpha$-Favard space corresponding to the semigroup $(\e^{-\omega t}T(t))_{t\geq 0}$ for some $\omega > \omega(T)$. 
	\end{definition}
 Note that this definition yields the same space for any such choice of $\omega$;  {for more on this topic we refer to Engel and Nagel \cite[Chapter~II.5]{engel_one-parameter_2000}}.
 
	We now quote results about Favard spaces needed in the sequel.
	\begin{remark}\label{rem:favard}
	For $A$ generating a $\mathrm{C}_0$-semigroup $T$ the following hold:
	\begin{enumerate}[leftmargin=3\parindent,rightmargin=3\parindent]
		\item If $\alpha > \beta$, we have 
		\begin{equation*}
			X_\alpha \subset F_\alpha \hookrightarrow X_\beta \subset F_\beta,
		\end{equation*}
		 {cf.~\cite[Proposition~II.5.14, Proposition~II.5.33]{engel_one-parameter_2000}; note that $X_\alpha$ is defined alternatively there.}
		\item If $T$ is exponentially stable, the $\alpha$-Favard norm with $0 < \alpha \leq 1$ is equivalent to
		\begin{equation*}
			\nnnorm{x}_{F_\alpha} \coloneq \sup_{\lambda>0}{\lambda^{\alpha}\norm{AR(\lambda,A)x}_X},
		\end{equation*}\label{rem:favardproperties}
		cf.~\cite[Proposition~II.5.12]{engel_one-parameter_2000}.
	\end{enumerate}
	\end{remark}
 We give a characterization of the property that $B$ maps into extrapolation spaces of the form $X_{-\beta}$ for $\beta \in (0,1)$ by admissibility of $B$, providing a counterpart to Lemma~\ref{lem:interpolation} above. In this context we note that similar results have appeared in \cite{maragh_admissible_2014}.
 \begin{lemma}\label{lem:resolventfractionalspace}
  {Let $A$ generate an analytic semigroup $T$ on $X$.}
    Let $\beta \in (0,1)$ and assume that for all $p > \frac{1}{1-\beta}$ the resolvent condition in the $p$\kern+.06em{}-\kern-.15em{}Weiss condition holds for $A$ and a given $B \in \LL(U,X_{-1})$, i.e.\ 
    \begin{equation}
        {\exists \alpha > \omega(T)\colon} \sup_{\Re \lambda > \alpha} (\Re \lambda)^{1/p} \norm{R(\lambda,A)B}_{\LL(U,X)} < \infty. \label{eq:resolventcondition}
    \end{equation}
    Then $B$ is bounded as a mapping $B\colon U \to X_{-\alpha}$ for all $\alpha > \beta$.
    \end{lemma}
    \begin{proof}
    Again we assume without loss of generality that $T$ is exponentially stable.
	For a fixed $q > \frac{1}{1-\beta}$ we have by the resolvent condition  {\eqref{eq:resolventcondition}} that 
	\begin{equation*}
		\norm{R(\lambda,A)Bu_0}_X \leq K ({ \Re \lambda})^{-1/q} \norm{u_0}_U, \qquad  {\Re\lambda > 0}.
	\end{equation*}
   Using the  {extension $A_{-1}$} of the generator $A$, this implies
   \begin{equation}
    (\Re\lambda)^{1/q} \norm{A R(\lambda,A)Bu_0}_{X_{-1}} \leq K \norm{u_0}_U \label{eq:resolventX-1}
   \end{equation}
   for all $\lambda \in \C$ with $\Re \lambda > 0$.
	By exponential stability of $T$  {and $0 < \gamma < 1$, Item \ref{rem:favardproperties} in Remark \ref{rem:favard} implies that} we can consider the norm
	\begin{equation*}
		\nnnorm{x}_{F_{-\gamma}} = \sup_{\lambda>0}{\lambda^{1-\gamma}\norm{AR(\lambda,A)x}_{X_{-1}}}, \qquad x \in X_{-1},
	\end{equation*}
	which is equivalent to the usual norm on $F_{-\gamma} \subset X_{-1}$. Recall that for negative parameters the Favard space was defined in terms of the extended generator on $X_{-1}$. Using
	\eqref{eq:resolventX-1} we get
	\begin{equation*}
	\nnnorm{Bu_0}_{F_{-(1-1/q)}} = \sup_{\lambda > 0} \lambda^{1 - (1 - 1/q)} \norm{AR(\lambda,A) Bu_0}_{X_{-1}} \leq K \norm{u_0}_U.
	\end{equation*}
	Relating this to abstract Sobolev spaces, notice that for any choice of  {$\alpha \in (\beta,1]$}
    	there exists a $q > 1$ such that $\alpha > 1 - \frac{1}{q} > \beta$. In fact, a possible choice is $q = (1 - \frac{\alpha + \beta}{2})^{-1}$. Embedding relations between 
	Favard and abstract Sobolev spaces from Remark~\ref{rem:favard} then imply $F_{-(1-1/q)} \hookrightarrow X_{-\alpha}$.
	As a consequence, we also have
	\begin{equation*}
		\norm{Bu_0}_{X_{-\alpha}} \leq K_1 \norm{Bu_0}_{F_{-(1-1/q)}} \leq K_2
			\norm{u_0}_U
	\end{equation*}
	for all $\alpha > \beta$. Hence we have shown that $B \in \LL(U,X_{-\alpha})$.
	\end{proof}
   {Note that Lemma~\ref{lem:resolventfractionalspace} is in 
  particular applicable if $B$ is even infinite-time $p$\kern+.06em{}-admissible for all $p > \frac{1}{1-\beta}$. Indeed, specifying the input function $u = \e^{-\lambda(\,\cdot\,)}u_0$ for $u_0 \in U$ and using the representation of the resolvent as the Laplace transform of the semigroup, the definition of infinite-time $p$\kern+.06em{}-admissibility yields \begin{equation*}
		\norm{R(\lambda,A)Bu_0}_X \leq K (\Re \lambda)^{-1/p} \norm{u_0}_U, \qquad \lambda \in \rho(A).
	\end{equation*}}
  {
 The next theorem combines Lemmas~\ref{lem:interpolation} and \ref{lem:resolventfractionalspace} into a characterization of
 $p$\kern+.06em{}-admissibility of $B$ for $p$ greater than a threshold in terms of 
 boundedness of $B$ into spaces of type $X_{-\alpha}$.}
 \begin{theorem} {
 Let $A$ be the generator of an analytic and exponentially stable semigroup $T$ on $X$ and let $B \in \LL(U,X_{-1})$ be a control operator.
 Then the following are equivalent for $\beta \in (0,1)$:
    \begin{enumerate}[leftmargin=3\parindent,rightmargin=3\parindent,label=\textup{(}\alph*\textup{)}]
    \item $B$ is finite-time $\Ll^p$-admissible for all $p > \frac{1}{1-\beta}$.
    \item $B$ is a bounded map into $X_{-\alpha}$ for all $\alpha > \beta$.
    \end{enumerate}
     }
     \begin{proof} {
    If $B\in \LL(U,X_{-1})$ is $\Ll^p$-admissible for all $p > \frac{1}{1-\beta}$, then $B$ is also infinite-time $\Ll^p$-admissible for all $p > \frac{1}{1-\beta}$ due to the assumption of exponential stability. Since in this case the resolvent condition \eqref{eq:resolventcondition} holds for all $p > \frac{1}{1-\beta}$, we have $B \in \LL(U,X_{-\alpha})$ for
    $\alpha > \beta$ by Lemma \ref{lem:resolventfractionalspace}. The implication $\text{(b)} \Rightarrow \text{(a)}$ directly follows from an application of 
    Lemma~\ref{lem:interpolation}.}
     \end{proof}
 \end{theorem}
	The statement in the previous  {theorem} does not hold in the endpoint case $p = \frac{1}{1-\beta}$ and $\alpha=\beta$, as 
    the following result shows.

    \begin{proposition}\label{rem:counterexample_interpolation}
    For any $p\in(1,\infty)$ there exists an analytic semigroup generator on $X=\ell^{2}$ and a $p$\kern+.06em{}-admissible $B:\C\to X_{-1}$
	that is not in $\LL(\C,X_{-1/p'})$. 
    \end{proposition}   
	\begin{proof}
    We adapt an example given in
	\cite[Example~5.3.11]{tucsnak_observation_2009} to $p\neq 2$.
	In fact, we let $1 < p < \infty$ be arbitrary. As in \cite[Example 5.3.11]{tucsnak_observation_2009}, we consider the semigroup generated by the diagonal operator $A$, that is, $Ae_{k}=\lambda_{k}e_{k}$, $k\in\mathbb{N}$; with $\{e_{k}\}_{k\in\mathbb{N}}$ referring to the canonical basis and with eigenvalue sequence defined by $\lambda_k = - 2^k$ for $k\in\N$. 
 
 Due to the identification of $\LL(\C,X_{-1})$ with the space $X_{-1}$ itself, the claim amounts to existence of a sequence $b \in X_{-1}$ such that $b$ is $p$\kern+.06em{}-admissible but fails to be an element of $X_{-1/p'}$.
		We consider the cases $1 < p \leq 2$ and $2 < p < \infty$ separately and use
		the appropriate Carleson measure criteria from \cite[Theorem~3.2, Theorem~3.5]{jacob_applications_2014}; see also Theorem~\ref{thm:JPP_sequences} below.

		Let $1 < p \leq 2$ and set $b = \{2^{k/p'}\}_{k\in\N}$. This choice of $b$
		represents an element of $X_{-1}$ as 
		\begin{equation*}
			\sum_{k=1}^\infty {\frac{\abs{b_k}^2}{\abs{\lambda_k}^2}} = \sum_{k=1}^\infty
				2^{2k/p'}\, 2^{-2k} = \sum_{k=1}^\infty 2^{-2k/p} < \infty.
		\end{equation*}
		{ 
        To prove $p$\kern+.06em{}-admissibility using the Carleson measure criterion we need to show that there exists $K > 0$ such that 
		$\mu(Q_I) \leq K \abs{I}^{2/p'}$, where $\mu = \sum_k \,\abs{b_k}^2 \delta_{-\lambda_k}$
		and $Q_I$ is the Carleson square defined for an interval $I$ on
		the imaginary axis symmetric about $0$ (an interval of the form $\I [-a,a]$ for some $a > 0$) by 
		\begin{equation*}
			Q_I \coloneq \{ z\in\C : \I \Im z \in I, 0 <
		\Re z < \abs{I} \},
		\end{equation*}}
  see also Figure~\ref{fig:carleson} below.
  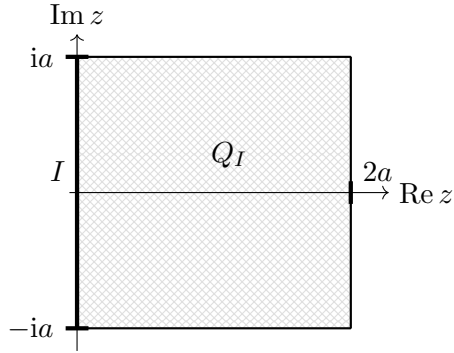
\begin{figure}[ht]
  \centering 
        \begin{tikzpicture}
            \draw[pattern=crosshatch, pattern color=lightgray!50] (0,1.8) rectangle (3.6,-1.8);
            \draw[->] (-0.1,0) -- (4.1,0) node[right]{$\Re z$}
            node[above=0.2cm,midway]{$Q_I$};
            \draw[->] (0,-2.1) -- (0,2.1) node[above]{$\Im z$};
            \draw[ultra thick] (0,-1.8) -- (0,1.8) node[above left,midway]{$I$}; 
            \draw[ultra thick] (0.15,-1.8) -- (-0.15,-1.8) node[left]{$-\I a$};
            \draw[ultra thick] (0.15,1.8) -- (-0.15,1.8) node[left]{$\I a$};
            \draw[thick] (0,-1.8) -- (3.6,-1.8);
            \draw[thick] (0,1.8) -- (3.6,1.8);
            \draw[thick] (3.6,1.8) -- (3.6,-1.8);
            \draw[ultra thick] (3.6,-0.15) -- (3.6,0.15) node[above right, midway]{$2a$};
        \end{tikzpicture}
        \caption{ The Carleson square $Q_I$ for the interval
        $I = \I [-a,a]$}
        \label{fig:carleson}
  \end{figure}
  
		We therefore proceed by showing that there exists $K > 0$ with
		\begin{equation*}
			\sum_{k\in\N} 2^{2k/p'} \delta_{2^k}((0,\abs{I})) \leq K\abs{I}^{2/p'}
		\end{equation*}
		for all choices of $I$ as above. Since $\abs{I} \geq 2^{
		\floor{\log_2 \abs{I}}}$ we set $m = \floor{\log_2 \abs{I}}$
		and get 
			\begin{align*}
				\sum_{k \in \N} 2^{2k/p'} \delta_{2^k}((0,\abs{I}))
				&= \sum_{k=1}^m ( 2^{2/p'} )^{ k} =
				\frac{2^{2/p'}}{2^{2/p'} - 1} ((2^m)^{2/p'} - 1)\\
				&\leq K_{p'} \abs{I}^{2/p'},
			\end{align*}
		which implies $p$\kern+.06em{}-admissibility of $b$.

		Now, if $b$ were in $X_{-1/p'}$, then we would have
		$(-A)^{-1/p'} b \in X$. 
        But due to $( (-A)^{-1/p'} b )_k \equiv 1$, this condition is violated and
		it follows that $b \not\in \LL(\C,X_{-1/p'})$.

		In the case $p > 2$ we consider the same semigroup, but now we let $b$ be the
		control operator represented by  
		$b_k = k^{-1/2}2^{k/p'}$, $k\in\N$. We have $b \in X_{-1}$ since 
			\begin{equation*}
				\sum_{k=1}^\infty \frac{\abs{b_k}^2}{\abs{\lambda_k}^2} = \sum_{k=1}^\infty
				k^{-1} \,2^{2k/p'}\, 2^{-2k} \leq \sum_{k=1}^\infty 2^{-2k/p} < \infty.
			\end{equation*}
		In this case, the Carleson measure criterion reads 
		\begin{equation*}
				\{2^{-2n/p'} \mu(S_n)\}_{n\in\N} \in \ell^{p/(p-2)}(\N)
		\end{equation*}
		where
		$\mu = \sum_{k\in\N} \,\abs{b_k}^2 \delta_{-\lambda_k}$ and 
		$S_n = \{ z \in \C : \Re z \in (2^{n-1}, 2^n] \}$. 

		For this choice of semigroup and control operator the condition means that
			\begin{equation*}
				\sum_{k=1}^\infty \,\bigl|k^{-1/2}\,2^{k/p'}\bigr|^{\frac{2p}{p-2}}\cdot 2^{-
				\frac{2k}{p'}\frac{p}{p-2}} < \infty,
			\end{equation*}
		which holds true here as 
			\begin{equation*}
				\sum_{k=1}^\infty \,\bigl|{k^{-1/2}\, 2^{k/p'}}\bigr|^{\frac{2p}{p-2}} \cdot2^{-
				\frac{2k}{p'}\frac{p}{p-2}} = \sum_{k=1}^\infty {k^{-\frac{p}{p-2}}}
				< \infty.
			\end{equation*}
		However, 
			\begin{equation*}
			\norm{(-A)^{-1/p'} b}_{\ell^2} = \sum_{k=1}^\infty 
			2^{-2k/p'} \,\abs{k^{-1/2}\, 2^{k/p'}}^2 = \sum_{k=1}^\infty k^{-1}
			\end{equation*}
		with a divergent sum; thereby showing that $b$ is unbounded as an operator $\C \to X_{-1/p'}$.
	\end{proof}
	
	\subsection{Example: admissibility for the heat equation via interpolation}\label{sec:dirichletheateq}
    
	As indicated before, we are interested in $\Ll^p$-admissibility for the heat equation on 
	the spatial domain $[0,1]$ with Dirichlet control at the boundary point $1$, given through the system's equations
		\begin{equation*}\left\{
				\begin{aligned} 
					x_t &= x_{\xi\xi} & &\text{in } (0,1) \times (0,\infty),\ \\
					x &= 0 & &\text{on } \{\xi  = 0\} \times (0,\infty), \\
					x &= u & &\text{on } \{\xi  = 1\} \times (0,\infty), \\
					x &= x_0 & &\text{on } (0,1) \times \{t = 0\},
				\end{aligned} \right.
		\end{equation*}
	 with initial heat distribution given by $x_0$ and heat injection 
	with magnitude $u(t)$ at the point $1$.
 
 In this simple example, all properties of the system can basically be derived from the solution formulae given through the spectral decomposition. Yet, it serves as the most natural and well-known example where $2$-admissibility is not satisfied if the state space is chosen to be $X=\Ll^{2}([0,1])$, see also \cite{lasiecka_control_2000}.  We formulate this as a boundary control system by defining
	\[\mathfrak{A} \colon X \supset \mathcal{D}(\mathfrak{A}) \to X,\quad f \mapsto f''\]
    on the 
	domain $\mathcal{D}(\mathfrak{A}) = \{ f \in \Hh^2([0,1]) : f(0) = 0 \}$. The
	abstract boundary operator is then given by \[\mathfrak{B} = \delta_1 \colon X \supset \mathcal{D}(\mathfrak{A}) \to U,\quad
	f \mapsto f(1).\] 
     {The Sobolev embedding theorem (see e.g.\ \cite[Theorem 4.12]{adamsSobolevSpaces2003}) implies} that each $\Hh^2$ function in spatial dimension $n=1$ is continuous, ensuring that $\mathfrak{B}$ is well-defined. 
    We want to translate this system into state space form to make 
	statements about admissibility. Define the operator $A$ by \[A = \mathfrak{A}|_{\ker \mathfrak{B}}\colon 
	\mathcal{D}(A) \to X,\quad f\mapsto f''\]
	on the domain $\mathcal{D}(A) = \{ f \in \Hh^2([0,1]) : f(0) = f(1) = 0 \}$.
	Using the identity \eqref{eq:hilbertcontrolop}---as both $X$ and $U$ are Hilbert spaces---one integrates by parts 
	twice to find that $\mathfrak{B} x \cdot B^*y = - x(1)\,y'(1)$ for all $x \in \mathcal{D}(\mathfrak{A})$ and $y\in \mathcal{D}(
	A)$. Thus we can identify that $B$ is given by
	\[B = \delta'_1 \!\colon \C \to X_{-1},\quad z \mapsto z\delta'_1,\] where $\delta'_1(f) = -f'(1)$.

	To proceed, we consider the orthonormal basis $\{ e_n : n\in\N \}$
	with $e_n(\xi) = \sqrt{2} \sin(n\pi\xi)$. Then the diagonal representation $Ae_n = \lambda_n e_n$ holds with $\lambda_n = -n^2\pi^2$.
	The control operator $B$ can be represented by the sequence $\{ b_n \}_{n\in\N}$ for $ b_n = (-1)^n\sqrt{2}n\pi$ as in \cite[Example~5.1]{Jacob_2018}.
    
	Using the representation of the system in diagonal form, we show that $b = \{b_n\}_{n\in\N}$
	is an element of the interpolation space $X_{-\beta}$ for $\beta > 3/4$. By Lemma~\ref{lem:interpolation}, this then
	will imply that the control operator is $p$\kern+.06em{}-admissible for all $p > \frac{1}{1 - 3/4} =4$. It remains to show that
	\[
		\sum_{n=1}^\infty \frac{\abs{b_n}^2}{\abs{\lambda_n}^{2\beta}} < \infty
	\]
	for all $\beta > \frac 34$. We have
	\begin{equation*}
		\sum_{n\in\N} \frac{\abs{b_n}^2}{\abs{\lambda_n}^{2\beta}} =
		\sum_{n\in\N} \frac{2 n^2 \pi^2}{\pi^{4\beta} 
		n^{4\beta}} = \frac{2}{\pi^{4\beta - 2}} \sum_{k\in\N} \frac{1}{n^{4\beta - 2}},
	\end{equation*}
	which is finite for  $\beta > \frac 34$. Invoking Lemma~\ref{lem:interpolation} then implies $p$\kern+.06em{}-admissibility for all $p > 4$.
    \begin{remark}
	One can also use the interpolation space argument to show that the Dirichlet controlled
	heat equation is admissible for $p > 4$ in multiple spatial dimensions. To show this, we follow \cite{schwenninger_input--state_2020} and 
	fix a dimension $n \geq 2$ and a domain $\Omega \subset \R^n$ with sufficiently smooth 
	boundary $\partial \Omega$. The heat equation with Dirichlet boundary control is represented by the system
	\begin{equation*}\left\{
        \begin{aligned}
		\dot x(t) &= \Delta x(t) \\
		\gamma_0 x &= u,\\
		x(0) &= x_0 \in X
        \end{aligned}
        \right.
	\end{equation*}
	on $X=\Ll^2(\Omega)$ with input space $U = \Ll^2(\partial \Omega)$, where 
    $\gamma_0$ is the Dirichlet trace operator on $\partial \Omega$. Relating the above to abstract boundary control systems, we set 
	\begin{equation*}
		\mathfrak{A} = \Delta,\qquad \mathcal{D}(\mathfrak{A}) =\Hh^2(\Omega) + D\Ll^2(\partial\Omega),\qquad
		\mathfrak{B} = \gamma_0,
	\end{equation*}
	where  {$D \in \LL(\Ll^2(\partial \Omega),\Ll^2(\Omega))$ denotes the Dirichlet map which maps $g \in \Ll^2(\partial \Omega)$ to the unique solution $h \in \Ll^2(\Omega)$ of 
 \[
        \Delta h = 0,\qquad \gamma_0 h = g;
 \]
 see 
 \cite[Proposition~10.6.6]{tucsnak_observation_2009}. Since by 
 \cite[Proposition~10.6.4]{tucsnak_observation_2009}} the composition $\gamma_0 D$ yields the identity on $\Ll^2(\partial\Omega)$, we have $\ker \mathfrak{B} = 
	\Hh^2(\Omega) \cap \Hh^1_0(\Omega)$. The general rewriting procedure for boundary control systems leads to the self-adjoint operator
	\begin{equation*}
	A = \mathfrak{A}|_{\ker\mathfrak{B}} = \Delta,\qquad \mathcal{D}(A) = \Hh^2(\Omega) \cap \Hh^1_0(\Omega).
	\end{equation*}
    As $X$ and $U$ are Hilbert spaces, we use
	the formula \eqref{eq:hilbertcontrolop} and Green's first identity to deduce that 
	\begin{equation*}
		\spr{\mathfrak{A}x}{y}_X - \spr{x}{Ay}_X = \spr{\mathfrak{B}x}{B^*y}_U 
    = - \spr{\mathfrak{B}x}{\tfrac{\partial y}{\partial \nu}}_{U},
	\end{equation*}
	whereupon we find that $B^*y$ is the normal derivative operator 
    $\frac{\partial y}{\partial \nu}$ on $\partial \Omega$. By \cite[Lemma~3.1.1]{lasiecka_control_2000}, 
    $B^*$ is bounded
	 $\Hh^\beta(\Omega) \cap \Hh^1_0(\Omega) \to \Ll^2(\partial\Omega)$ 
    for all $\beta > \frac 32$.
	Hence, for all $\beta > 3/4$ this yields \[B^*\in \LL(X_{\beta}, U^*) = \LL(\Hh^{2\beta}(\Omega) \cap \Hh^1_0(\Omega), \Ll^2(\partial \Omega)), \qquad B^*\colon f \mapsto \biggl.\frac{\partial f}{\partial \nu}\biggr|_{\partial\Omega}.\] 
    	By duality---as $X = \Ll^2(\Omega)$ is reflexive---we conclude that $(-A)^{-\beta}B$
	is bounded as an operator $U \to X$ for all $\beta > 3/4$. Lemma~\ref{lem:interpolation} then implies that $B$ is $p$\kern+.06em{}-admissible for all
	$p>4$.
    \end{remark}
	\begin{remark}[{Notes on Neumann control}]
    Although it does not appear using the notion of \emph{admissibility} explicitly,
    we remark that the heat equation on the state space $X = \Ll^2(\Omega)$
    with Neumann control acting 
    via the boundary---and inputs in $U = \Ll^2(\partial \Omega)$---is featured in the book by Lasiecka and Triggiani
    \cite[Remark~3.3.1.4]{lasiecka_control_2000} (and the result is folklore). There the control operator $B$
    is computed and shown to satisfy $(-A)^{-(1/4 + \epsilon)} B \in \LL(U,X)$
    for any $\epsilon > 0$. This allows
    the application of the interpolation space argument as in Section~\ref{sec:interpolation}.
    
    Furthermore, $p$\kern+.06em{}-admissibility of the Neumann controlled heat 
    equation is characterized by an embedding argument 
    due to Haak and Kunstmann, see \cite[Proposition~2.4]{haakWeightedAdmissibilityWellposedness2007}.
    We also refer to Byrnes, Gilliam, Shubov and Weiss \cite{byrnesRegularLinearSystems2002}
    for the case $p = 2$.
    Denoting the Neumann Laplacian on $\Omega \subset \R^n$ by $\Delta_{\mathrm{N}}$, it follows from
    Haak and Kunstmann's results that
    the control operator $B = \gamma_0'$ with input space $\Ll^2(\partial \Omega)$---resulting from the Neumann controlled heat equation
    on $X = \Ll^2(\Omega)$---is $p$\kern+.06em{}-admissible for $A = -\Delta_{\mathrm{N}} + I$ if
    and only if $p \geq \frac 43$. This is a special case of a more general result
    they obtain in the context of weighted admissibility
    on state spaces such as $\Ll^q(\Omega)$.
    In the unweighted case that we consider here, the result yields $p$\kern+.06em{}-admissibility whenever $U$ embeds into the Besov space $\mathrm{B}_{q,\infty}^{2/p - 1 - 1/q}(\partial \Omega)$.  {We employ the Besov space embedding\[
    \mathrm{B}_{2,2}^{0} \hookrightarrow \mathrm{B}_{2,\infty}^{-s}, \qquad s \geq 0,
    \]
    see e.g.\ \cite[Section~2.3.2, Remark~2]{triebelTheoryFunctionSpaces1992} or \cite[Section~2.3]{schneiderSobolevBesovRegularity2021a}, such that }in the case of 
    the Neumann controlled heat equation on $X = \Ll^2(\Omega)$ with $U = \Ll^2(\partial \Omega)$ one arrives at the condition
    \[
        U = \mathrm{B}_{2,2}^0(\partial \Omega) \hookrightarrow \mathrm{B}^{2/p - 3/2}_{2,\infty}
        (\partial \Omega)
    \]
    which is satisfied if and only if $p \geq \frac 43$, as claimed.
    \end{remark}
	\section{Characterizations for multiplier-like systems}
	We are interested in testing $p$\kern+.06em{}-admissibility also in cases where the system is not necessarily of 
	diagonal form. In this section, we formulate extensions of
	Laplace--Carleson type results to normal semigroups and multiplication 
	semigroups.
    For this, recall the following result due to Jacob, Partington and Pott for diagonal semigroups.
    \begin{theorem}[{\cite[Theorem~2.1]{jacob_applications_2014}}]\label{thm:JPP_L-C}
    Let $A$ generate a diagonal semigroup on $X=\ell^{q}(\Z)$\footnote{More generally, one can assume that the semigroup is diagonal with respect to a $q$-Riesz basis on an arbitrary Banach space $X$.} with eigenvalues given by $\{\lambda_{k}\}_{k\in\mathbb{Z}}$. Let $p\in[1,\infty)$. Then $B=\{b_{k}\}_{k\in\Z}\in \LL(\C,X_{-1})$ is infinite-time $\Ll^{p}$-admissible if and only if the Laplace transform 
    \begin{equation*}
        \LLL\colon \Ll^p([0,\infty)) \to \Ll^q(\C_+,\mu),\quad f\mapsto \left(z\mapsto\int_{0}^{\infty}\e^{-zt}f(t)\,\mathrm{d}t\right)
    \end{equation*}
    is a bounded operator for the measure
   \[\mu = \sum_{k\in\Z} \,\abs{b_k}^q \delta_{-\lambda_k}.
   \]
    \end{theorem}
    Note that by using the special form of a diagonal semigroup the proof of the above result is rather direct and the difficulty in assessing admissibility only transfers to checking the boundedness of the Laplace--Carleson embedding. This is where arguments from harmonic analysis, and in particular from the theory of Carleson measures,  {take effect}---at least under further assumptions on the range of $p$ and the location of the eigenvalues (in sectors or strips). The latter situations reflect whether the diagonal semigroup is analytic or even a 
    group. 
    In those cases, admissibility indeed reduces to elementary conditions on certain sequences lying in corresponding $\ell^{r}$-spaces, such as employed in Section~\ref{sec:interpolation} above and Section~\ref{sec:4admissibility} below.
    
    For convenience, we recall the admissibility results for 
    diagonal semigroups on $\ell^q$ spaces from \cite{jacob_applications_2014}.
    \begin{proposition}[{\cite[Theorem~3.2, Theorem~3.5]{jacob_applications_2014}}]\label{thm:JPP_sequences}
        Fix $1 \leq q < \infty$ and assume that $A\colon \ell^q(\Z) \supset \mathcal{D}(A) \to \ell^q(\Z)$ generates a diagonal semigroup on $X = \ell^q(\Z)$ with eigenvalues
        $\{\lambda_k\}_{k\in\Z}$ lying in the left half plane such that $-\lambda_k \in S_\theta$ for some angle $\theta \in (0,\pi/2)$ and all $k \in \Z$. Consider the operator $B \in \LL(\C,X_{-1})$ represented by the sequence $\{ b_k \}_{k\in\N}$. 

        { 
        If $1 < p \leq q$, then the following statements are equivalent:
        \begin{enumerate}[ leftmargin=3\parindent,rightmargin=3\parindent]
        \item $B$ is $p$\kern+.06em{}-admissible.
        \item There exists $K > 0$ such that \[\mu(Q_{I}) \leq K\abs{I}^{q/p'}\] holds for all intervals $I \subset \I \R$ symmetric about $0$; i.e.\ intervals of the form $\mathrm{i} [-a,a]$ for some $a>0$, where $Q_I \coloneq \{ z\in\C : \I \Im z \in I, 0 <
		\Re z < \abs{I} \}$.
        \item There exists $K > 0$ such that $\norm{R(z,A)B}_X \leq K z^{-1/p}$ for all $z > 0$.
        \end{enumerate}
        Here $Q_{I} = \{ z\in\C : \I \Im z \in I, 0 <
		\Re z < \abs{I} \}$.
  
        If instead $q < p < \infty$, the following are equivalent:
        \begin{enumerate}[leftmargin=3\parindent,rightmargin=3\parindent]
        \item $B$ is $p$\kern+.06em{}-admissible.
        \item $\{ 2^{-nq/p'} \mu(S_n) \}_{k\in\Z} \in \ell^{p/(p-q)}(\Z)$.
        \item $\{ 2^{n/p} \norm{R(2^n,A)B}_{\ell^{q}} \}_{k\in\Z} \in \ell^{p/(p-q)}(\Z)$.
        \end{enumerate}
        Here $\mu = \sum_{k\in\N} \,\abs{b_k}^2 \delta_{-\lambda_k}$ and 
		$S_n = \{ z \in \C : \Re z \in (2^{n-1}, 2^n] \}$}.
    \end{proposition}
    Thus, in this section we are looking for (more general) measures $\mu$ such that the equivalence
	\begin{equation*}
		\LLL \in \LL(\Ll^p([0,\infty)),\Ll^q(\C_+,\mu))
		\Leftrightarrow 
		\norm*{\int_0^\infty T(s)Bu(s) \dd s}_X \leq C \norm{u}_{\Ll^p([0,\infty))}
	\end{equation*}
	holds for normal semigroups and multiplication semigroups.
    \subsection{Non-admissibility for a Dirichlet controlled heat equation}
    \label{sec:4admissibility}
	As an application of the Laplace--Carleson method from \cite{jacob_applications_2014} we want to show that the Dirichlet controlled heat equation in one spatial 
	dimension as above is not $\Ll^p$-admissible for any $p\leq4$. 
 
    It is enough to show the claim for $p=4$. To show that Dirichlet boundary control does not lead to a $4$-admissible control operator for this system, we apply the necessary and sufficient condition given by the Laplace--Carleson
    method from \cite[Theorem~3.5]{jacob_applications_2014}, as above in Proposition~\ref{rem:counterexample_interpolation} and  {Proposition}~\ref{thm:JPP_sequences}. Keeping the same notation, i.e.\ $\mu = \sum_{k\in\N} \,\abs{b_k}^2 \delta_{-\lambda_k}$ and 
		$S_n = \{ z \in \C : \Re z \in (2^{n-1}, 2^n] \}$, we insert $p = 4$ into the condition,
    which yields
    \[
        \text{$B$ $4$-admissible} \Leftrightarrow s \coloneq \{ 2^{-3n/2} \mu(S_n) \}_{n\in\Z} \in \ell^2(\Z).
    \]
    In showing that $s$ is not square-summable, we use the following expression for $\mu(S_n)$: Recall that $\lambda_k = -k^2\pi^2$ (\cite[Example~5.1]{Jacob_2018}). For $n \in \Z$ we have
		\begin{equation}
		\begin{aligned}
			2^{-2n/p'} \mu(S_n) &= 2^{-2n/p'} \sum_{k\in\N}\,\abs{\sqrt{2} k \pi}^2 
			\delta_{k^2 \pi^2}
			(\{ z \in \C : \Re z \in (2^{n-1},2^n]\})\\
			&= 2^{1-2n/p'} \pi^2 \sum_{\crampedclap{k \in \N \cap \bigl(\!\frac{\sqrt{2^{n-1}}}{\pi}
			, \frac{\sqrt{2^{n}}}{\pi}
			\bigr]}} k^2. 
   \end{aligned}
   \label{eq:muS_n_first}    
		\end{equation}
     Combining \eqref{eq:muS_n_first} with the fact that $k \mapsto k^2$ is an increasing function on $[0,\infty)$, we estimate
    \[
     \mu(S_n) = 2\pi^2 \sum_{\crampedclap{k \in \N \cap \bigl(\!\frac{\sqrt{2^{n-1}}}{\pi}
			, \frac{\sqrt{2^{n}}}{\pi}
			\bigr]}} k^2 \geq  2\pi^2 \int_{\floor{{2^{(n-1)/2}}/{\pi}}}^{\floor{{2^{n/2}}/{\pi}}} k^2 \dd k = \frac{2\pi^2}{3} \Biggl( \floor*{\frac{2^{n/2}}{\pi}}^3 \!\!- \floor*{\frac{2^{(n-1)/2}}{\pi}}^3\,\Biggr).
    \]
    Then we get
    \[
        \mu(S_n) \geq \frac{2\pi^2}{3} \Biggl( \floor*{\frac{2^{n/2}}{\pi}}^3 \!\!- \floor*{\frac{2^{(n-1)/2}}{\pi}}^3\, \Biggr)
    \]
    which, in combination with the estimate $\floor{b}^3-\floor{a}^3 \geq
    (b-1)^3 - a^3$, gives
    \[
        \mu(S_n) \geq \frac{2\pi^2}{3}\biggl( \frac{(1-2^{-3/2})}{\pi^3}2^{{3n}/{2}} - \frac{3}{\pi^2}2^{{2n}/{2}} + \frac{3}{\pi}2^{{n}/{2}} - 1\biggr).
    \]
    Multiplying both sides with the prefactor $2^{-3n/2}$, we 
    have 
    \[
        2^{-{3n}/{2}}\mu(S_n) \geq \frac{2\pi^2}{3}\biggl(\frac{(1-2^{-3/2})}{\pi^3} - \frac{3}{\pi^2}2^{-{n}/{2}} + \frac{3}{\pi}2^{-{2n}/{2}} - 2^{-{3n}/{2}}\biggr),
    \]
    where the expression on the right-hand side does not converge to zero in the limit as  $n\to \infty$.
    Therefore the sequence $s$ is not square-summable and consequently $B$ is not $4$-admissible.
    
    \subsection{Normal semigroups}
	\begin{proposition}[Laplace--Carleson embeddings for normal semigroups]
 \label{prop:L-C_normal}
	Let $A$ generate a strongly continuous semigroup of normal operators
	on a Hilbert space $X$. Denote the spectral measure of $A$ by $E$ and the spectral measure of the extension $A_{-1}$ by $E_{-1}$.
    Consider a control operator $b \in X_{-1} \simeq 
	\LL(\C,X_{-1})$ and define the measure 
 \begin{equation}
 \mu(S) \coloneq \int_S (1 + \abs{z}^2)\dd \spr{E_{-1}(z)b}{b}_{X_{-1}} = (1 + \abs{z}^2) (E_{-1})_{b,b}(S)
 \end{equation} 
 on Borel subsets $S$ of $\C$. 
 
 Then $b$ is $\Ll^p$-admissible for $T$ if and only if the
	\kern-.06em{}Laplace transform is a bounded operator when considered as a map
	\begin{equation*}
		\LLL \colon \Ll^p([0,\infty)) \to \Ll^2(\sigma(-A),\mu).
	\end{equation*}
	\end{proposition}
	\begin{proof}
		By the spectral theorem applied to the normal operator $A_{-1}$, we have the representations
		\[
			(rI-A_{-1}) = \int_{\sigma(-A)} (r - z) \dd E_{-1}(z)\quad (r\in \C),\qquad T(t) = \int_{\sigma(-A)} \e^{-tz} \dd E_{-1}(z).
		\]
		Let $u$ be a locally integrable complex-valued function on $[0,\infty)$. To prove the claim, we show that
  \[
       \norm{\LLL u}_{\Ll^2(\sigma(-A),\mu)} \simeq \norm*{\int_0^\infty T(s)bu(s) \dd s}_X,
  \]
  where $a \simeq b$ means there exist $C_1, C_2 > 0$ such that $a \leq C_1 b$ and $b \leq C_2 a$.
  
  To begin the proof, we
  expand the expression on the right-hand side and get
		\begin{align*}
			\norm*{\int_0^\infty T(s) bu(s) \dd s}_X^2 &= 
        \norm*{\int_{\sigma(-A)} \int_0^\infty\e^{-sz} u(s) \dd s \dd E_{-1}(z) b}_X^2\\
        &= \norm*{(r I - A_{-1})\int_{\sigma(-A)}\int_0^\infty \e^{sz} u(s) \dd s \dd E_{-1}(z) b }_{X_{-1}}^2\\
        &= \norm*{\int_{\sigma(-A)}( r - z ) \int_0^\infty \e^{sz} u(s) \dd s \dd E_{-1}(z) b }_{X_{-1}}^2\\
        &= \norm*{\int_{\sigma(-A)}( r - z)\cdot(\LLL u)(z) \dd E_{-1}(z) b }_{X_{-1}}^2
		\end{align*}
  with $r \in \rho(A)$ being the element of $\rho(A)$ that was implicitly used to define
  the space $X_{-1}$.
  
		Given a measurable function $f : \sigma(A) \to \C$, one has the following norm representation:
		\[
			\norm*{\biggl(\int_{\sigma(-A)} f \dd E_{-1}\biggr)b}_{X_{-1}}^2 = \int_{\sigma(-A)} \abs{f}^2 \dd (E_{-1})_{b,b} {,}
		\]
		see e.g.\ \cite[Theorem~4.7]{conwayCourseFunctionalAnalysis2007}. Specifying the function $z \mapsto (r - z)\cdot (\LLL u)(z)$ as the function $f$ gives
		\begin{align*}
			\norm*{\int_{\sigma(-A)}( r - z)\cdot(\LLL u)(z) \dd E_{-1}(z) b }_{X_{-1}}^2
			&= \int_{\sigma(-A)}\abs{(r-z)\cdot(\LLL u)(z)}^2 \dd \spr{E_{-1}(z) b}{b}_{X_{-1}}\\
        &= \int_{\sigma(-A)}\abs{(\LLL u)(z)}^2 \abs{r-z}^2 \dd (E_{-1})_{b,b}(z)\\
        &\simeq \int_{\sigma(-A)}\abs{(\LLL u)(z)}^2 \abs{r-z}^2 \frac{1 + \abs{z}^2}{\abs{r - z}^2} \dd (E_{-1})_{b,b}(z)\\
        &= \norm{\LLL u}^2_{\Ll^2(\sigma(-A),\mu)},
		\end{align*}
  where we have used that $r \in \rho(A)$ implies that 
  \[
    \frac{1 + \abs{z}^2}{\abs{r - z}^2} \simeq C
  \]
  on $\sigma(A)$ for some $C > 0$.
		Then we also have
		\begin{equation*}
			\norm*{\int_0^\infty T(s)bu(s) \dd s}_X^2 \simeq \norm{\LLL u}_{\Ll^2(\sigma(-A),\mu)}^2.
		\end{equation*}
		Hence $p$\kern+.06em{}-admissibility of $b$ is equivalent to 
		boundedness of the Laplace transform.
	\end{proof}
 \begin{remark}\label{rem:extensionEbb}
     We note that in the literature the specific case of normal operators with discrete spectrum is treated in full detail, whereas the general case is often argued to follow similarly as for diagonal semigroups; cf.\ e.g.\ \cite{hansenOperatorCarlesonMeasure1991,staffans_well-posed_2005,weissTwoConjecturesAdmissibility1991} with the exception of \cite[Theorem~19]{weissPowerfulGeneralizationCarleson1999}, where an approximation argument is used.

     More precisely, in these references the measure $\mu$ is chosen to be $\spr{E(\,\cdot\,)b}{b}$ which is in fact only well-defined on bounded Borel sets for $b \in X_{-1}$ (see also the proof of Theorem~19 in \cite{weissPowerfulGeneralizationCarleson1999}). In Proposition~\ref{prop:L-C_normal} we circumvent this technical subtlety by employing a slightly different definition of $\mu$, which however coincides with the abovementioned choice for diagonal semigroups.
 \end{remark}
	\subsection{Conditions for multiplication semigroups}
	Comparable results to the ones in the previous subsection for normal semigroups
	on Hilbert spaces hold for semigroups generated by
	multiplication operators on $\Ll^q$ spaces.
 
 Let $(\Omega, \mu)$ be a $\sigma$-finite measure space.
 Given a measurable function $a\colon \Omega \to \C$, we let the 
	multiplier $M_a$ on $\Ll^q(\Omega)$, $1 < q < \infty$, be defined by
		\[	
			g \mapsto M_a g = ag,\qquad g \in \mathcal{D}(M_f) = 
			\{ g \in \Ll^q(\Omega) : ag \in \Ll^q(\Omega)\}.
		\]
	For multiplication semigroups, the extrapolation and interpolation spaces
	have convenient forms, see e.g.\  {\cite[Example~II.5.7]{engel_one-parameter_2000}}.
		 Suppose that $a$ is complex-valued
		 with $\esssup_\Omega \Re a < 0$. The 
		semigroup generated by the multiplication operator $M_a$ on $X$ is
		given by $(T(t))_{t\geq 0}$, where $T(t)f = \e^{ta}f$.
		Moreover, for $n\in \Z$, the abstract Sobolev spaces $X_n$ are given by the weighted Lebesgue spaces $X_n = \Ll^q(\Omega, \abs{a}^{nq}\mu)$.
	
	\begin{lemma}
		Let $1 < q < \infty$ and assume that the multiplier $M_a$ generates a strongly continuous 
		semigroup $T$ on $\Ll^q(\Omega,\mu)$ with $\Re a < 0$ $\mu$-a.e. \kern-.2em{}on $\Omega$.
		Then we have that 
			\[
				M_{\abs{\Re a}^{-1/q}} b \in \Ll^q(\Omega,\mu)  \text{ implies that $b$ is 
				$\Ll^{q'}$-admissible for $T$.}
			\]
	\end{lemma}
	\begin{proof}
		By the functional calculus for multiplication operators, the semigroup $T$ is 
		given by $(T(t)f)(s) = \e^{sa(t)}f(s)$. Therefore we need to show that there
		exists a constant $M > 0$ such that 
			\[
				\norm*{\int_0^\infty \e^{sa} b u(s) \dd s}_{\Ll^q(\Omega,\mu)}
				\leq M \norm{u}_{\Ll^{q'}([0,\infty))}
			\]
		holds for all $u \in \Ll^{q'}([0,\infty))$.

		Since 
			\[
				\norm*{\int_0^\infty \e^{sa} b u(s) \dd s}_{\Ll^q(\Omega,\mu)}^q =
				\int_\Omega\,\abs*{\int_0^\infty \e^{sa(t)} b(t) u(s) \dd s }^q \dd \mu(t),
			\]
		we can use Hölder's inequality to get 
			\begin{align*}
				&\int_\Omega \,\abs*{\int_0^\infty \e^{sa(t)} b(t) u(s) \dd s }^q \dd \mu(t)\\ 
				&\leq \int_\Omega \,\abs*{\left( \int_0^\infty \abs{\e^{sa(t)} b(t)}^q \dd s 
				 \right)^{1/q} \left( \int_0^\infty \abs{u(s)}^{q'} \dd s \right)^{1/{q'}}}^q \dd
				 \mu(t) \\
				&=\norm{u}_{\Ll^{q'}([0,\infty))}^q\int_\Omega \abs{b(t)}^q \int_0^\infty 
				\abs{\e^{sa(t)}}^q \dd s \dd \mu(t) \\
				&= \norm{u}_{\Ll^{q'}([0,\infty))}^q \int_\Omega \abs{b(t)}^q \int_0^\infty
				\e^{q s \Re(
        a(t))} \dd s \dd \mu(t) \\
				&= C(q) \norm{u}^q_{\Ll^{q'}([0,\infty))}\int_{\Omega}{\abs{b}^q}\
    \frac{1}{\abs{\Re a}} \dd \mu.
			\end{align*}
		This implies that $b$ is ${q'}$-admissible for $T$ if $b \in \Ll^q(\Omega, \abs{\Re a}^{-1}\mu)$.
	\end{proof}
	\begin{lemma}\label{lem:lapladm}
		Fix $1 < q < \infty$. Given a measurable complex-valued function $a$ on some $\sigma$-finite measure space $(\Omega,\mu)$ such that $\Re a$\kern-.05em{} is essentially bounded from above, 
		let
		$M_a$ be the corresponding multiplier on $\Ll^q(\Omega,\mu)$. Then $b \in \LL(\C, X_{-1}) \simeq X_{-1} = 
		\Ll^q(\Omega, \abs{a}^{-1}\mu)$ is $\Ll^{q'}$\!-admissible for $T$\kern-.1em{} if and only if the 
		Laplace transform induces a bounded map 
		\[
			\LLL\colon \Ll^{q'}([0,\infty)) \to \Ll^q(\Omega, \nu),\quad
			u \mapsto \LLL u = \int_0^\infty \e^{-t(\,\cdot\,)} u(t) \dd t,
		\]
		where the measure $\nu$ is defined by $\nu \coloneq (-a)_*\abs{b}^q\mu$.
	\end{lemma}
	\begin{proof}
	We have
	\begin{align*}
		\norm*{\int_0^\infty T(s) b u(s) \dd s}_{\Ll^q(\Omega,\mu)}^q &= 
		\int_\Omega\abs{b(t)}^q \abs*{\int_0^\infty \e^{sa(t)} u(s) \dd s}^q \dd \mu(t)\\
		&= \int_\Omega \abs{b(t)}^q \abs{(\LLL u)(-a(t))}^q \dd \mu(t) \\
		&= \int_\Omega \abs{(\LLL u)\circ (-a)}^q \dd (\abs{b}^q \mu) \\
		&=\norm{(\LLL u)}_{\Ll^q(\Omega,(-a)_*\abs{b}^q\mu)}^q,
	\end{align*}
	thus $b$ is $\Ll^{q'}$-admissible if and only if 
	\[
		\norm{(\LLL u)}_{\Ll^q(\Omega,\nu)} \leq \norm{u}_{\Ll^{q'}([0,\infty))},
	\]
	which is what we wanted to show.
	\end{proof}
	Due to invariance of admissibility under similarity transformations, this result extends to generators that are similar to a multiplier or unitarily equivalent to a multiplier, 
	and thus applies in particular to normal operators on a Hilbert space by a version of the 
	spectral theorem for unbounded operators. Furthermore, the diagonal case follows as a special case:
	\begin{example}[Diagonal semigroups]
	Let $A$ generate a diagonal semigroup on the state space $X = \ell^q(\N) = 
	\Ll^q(\N,\xi)$, where $\xi = \sum_k \delta_k$ is the counting measure on $\N$.
	Componentwise, this means that $(Ax)_k = \lambda_k x_k$ holds for all $k\in\N$ and all $x\in X$. 
	Interpreting $A$ as a multiplication operator on $X$ with symbol $a \colon \N \to \C,\ 
	k \mapsto \lambda_k$, we also identify the control operator $b \in X_{-1}$
    with a function $b\colon \N \to \C,\ k \mapsto b_k$.
	By Lemma~\ref{lem:lapladm}, the input element $b$ is admissible for $T$ if and only if the Laplace
	transform is a bounded map $\LLL\colon \Ll^p([0,\infty)) \to \Ll^q(\Omega, \nu)$ with measure $\nu$ given by
    \[
        \nu = (-a)_*\abs{b}^q\xi = \abs{b}^q \sum_{k} \delta_{-\lambda_k} = 
        \sum_k \,\abs{b_k}^q \delta_{-\lambda_k},
	\]
	which is the same measure as $\mu$ in Theorems~3.2 and 3.5 in \cite{jacob_applications_2014}.
	\end{example}
	\section{Examples with generators having non-discrete spectra}
	By considering generators
	with non-discrete spectrum, we can find examples to illustrate the usage of the Laplace--Carleson type condition 
	in the non-diagonal setting. Candidates are the Dirichlet or Neumann 
	Laplacians on a suitable domain, since their spectra need not be discrete.
	In passing, we note that the Neumann Laplacian on $\Omega$ has purely discrete spectrum if and
	only if the Rellich--Kondrachov compactness theorem holds for the domain $\Omega$. Hence
	a possible example would be the Neumann Laplacian on a sufficiently rough
	domain that fails the Rellich--Kondrachov embedding theorem, see e.g.\ 
    \cite{hempelEssentialSpectrumNeumann1991}. Another 
 operator with non-discrete spectrum, which we will consider in the sequel,
 is the Laplacian 
 on the whole space $\R^n$.
	\subsection{The Laplacian on the full space}
 \label{sec:laplacianfullspace}
 \begin{example} \label{ex:laplacianfullspace}
	Let $n \leq 3$ and consider the state space $X = \Ll^2(\R^n)$ and the 
	Laplacian $A = \Delta$ on $X$ with domain $\mathcal{D}(A) = \Hh^2(\R^n)$, which generates a $\mathrm{C}_0$-semigroup on $X$. 
 
 We choose the Dirac delta distribution $\delta_0$ as the control operator $b$ and claim that $b \in X_{-1}$ if and only if $n \leq 3$.
	\begin{proof}[Proof of the claim]
		By self-adjointness of $A$ the domains  $\mathcal{D}(A)$ and $\mathcal{D}(A^*)$ are identical. 
        Therefore the extrapolation
		space $X_{-1}$ is given by the dual of $\mathcal{D}(A) = \Hh^2(\R^n)$ with respect
		to the pivot space $X = \Ll^2(\R^n)$. As the spatial domain is the full space $\R^n$,
        we have
		$\Hh^2(\R^n) = \Hh^2_0(\R^n)$. Hence---by definition of Sobolev spaces with negative integer indices---it holds that $X_{-1} = \Hh^{-2}(\R^n)$.
		In the Fourier domain the space $\Hh^{-2}(\R^n)$ has the representation
			\[
				\Hh^{-2}(\R^n) = \{ u \in \mathcal{S}'(\R^n) : (1 + \abs{\,\cdot\,}^2)^{-1} \hat u\in
				\Ll^2(\R^n) \},
			\]
		where $\mathcal{S}'(\R^n)$ denotes the space of tempered distributions on $\R^n$.
		As the Fourier transform of $\delta_0$ computes to $\hat \delta_0 \equiv 1$ in the sense of tempered distributions, we have $\delta_0 \in X_{-1}$
		if and only if
			\[
				\int_{\R^n}\frac{1}{(1 + \abs{\xi}^2)^2} \dd \xi < \infty.
			\]
		By transforming to spherical coordinates in $\R^n$, this is equivalent to 
		finiteness of 
			\[
				\int_0^\infty \frac{r^{n-1}}{(1 + r^2)^2} \dd r,
			\]
		which is the case if and only if $n \leq 3$.
	\end{proof}
 
	The spectral measure $E_{g,g}$ corresponding 
    to the operator $A$ and some $g\in\Ll^2(\R^n)$ is computed in e.g.\ \cite[Equation~(7.34)]{teschlMathematicalMethodsQuantum2014} using the representation of the Laplacian by its Fourier symbol $\xi \mapsto -\abs{\xi}^2$ and is given by the Lebesgue measure $\lambda^n$ weighted with the density 
		\[
			\rho_{g,g}(z) = \chi_{(-\infty,0]}(z) \abs{z}^{\frac n2 - 1}\int_{\partial \ball{1}{n}{0}}
			\bigl|{\hat g(\abs{z}^{1/2}w)}\bigr|^2 \dd\hausdorff^{n-1}(w);
		\]
  in particular the spectrum of $A$ is purely absolutely continuous and $\sigma(A) =
  (-\infty,0]$.
  
	In the above formula $\partial\ball{1}{n}{0} \coloneq \{ y \in \R^n : \abs{y} = 1 \}$
	denotes the unit sphere in $\R^n$, $\hat g$ denotes the Fourier 
    transform of $g\in\Ll^2(\R^n)$ and $\hausdorff^{n-1}$ is the
	$(n-1)$-dimensional Hausdorff measure on $\R^n$. Note that this representation of $\rho_{g,g}$ is also valid in dimension $n = 1$ using the $0$-dimensional Hausdorff measure.
	
For our example of the control operator $b = \delta_0 \in X_{-1}$, we need to consider the extension of $E_{g,g}$ for $g \in \Hh^{-2}(\R^n) = X_{-1}$. For $g \in X_{-1}$ the density $(\rho_{-1})_{g,g}$ is then given by 
    \[
        (\rho_{-1})_{g,g}(z) = (1 + \abs{z}^2) \chi_{(-\infty,0]}(z) \abs{z}^{\frac n2 - 1}\int_{\partial \ball{1}{n}{0}}
			\abs*{\bigl(\fourier(R(1,A)g)\bigr)(\abs{z}^{1/2}w)}^2 \dd\hausdorff^{n-1}(w).
    \]
    We use $\hat\delta_0 \equiv 1$ to get
    \begin{align*}
        (\rho_{-1})_{\delta_0,\delta_0}(z) &= (1 + \abs{z}^2) \chi_{(-\infty,0]}(z) \abs{z}^{\frac n2 - 1}\int_{\partial \ball{1}{n}{0}}
			\abs*{\frac{1}{1 + \abs{\abs{z}^{1/2} w}^2}}^2 \dd\hausdorff^{n-1}(w)\\
   &= (1 + \abs{z}^2) \chi_{(-\infty,0]}(z) \abs{z}^{\frac n2 - 1}\int_{\partial \ball{1}{n}{0}}
			(1 + \abs{z})^{-2} \dd\hausdorff^{n-1}(w)\\
    &=\frac{1 + \abs{z}^2}{(1 + \abs{z})^{2}} \chi_{(-\infty,0]}(z) \abs{z}^{\frac n2 - 1}\cdot C(n)
    \end{align*}
    Since  $\frac{1}{2} \leq \frac{1 + \abs{z}^2}{(1 + \abs{z})^{2}} \leq 1$ 
    holds for $z \in \C$,
    we can reduce to the density 
    \begin{equation}
    (\rho_{-1})_{\delta_0,\delta_0}(z) = \chi_{(-\infty,0]}(z) \abs{z}^{\frac n2 - 1}\label{eq:reduceddensity}
    \end{equation}
    in applications of Proposition~\ref{prop:L-C_normal}.
    
    To make statements concerning finite-time $p$\kern+.06em{}-admissibility
	of $b$, we shift the generator of the semigroup such that the
 resulting semigroup is exponentially stable and apply the test from Proposition~\ref{prop:L-C_normal} for infinite-time admissibility. 
     
     Concretely, since $0 \in \sigma(A)$, to apply Proposition~\ref{prop:L-C_normal} we consider the shifted semigroup generated by the operator $A_1 = A - I$, such that $\sigma(A_1) = \sigma(A) - 1 = (-\infty, -1]$.
    Then by Proposition~\ref{prop:L-C_normal}, and using the expression for the density  $(\rho_{-1})_{\delta_0,\delta_0}$ from \eqref{eq:reduceddensity}, the proof of finite-time $\Ll^p$-admissibility of $b$ reduces
    to finding a Laplace transform bound of the form 
	\begin{equation}
		\int_{-\infty}^{-1} \,\abs*{\int_0^\infty \e^{t s} u(t) \dd t}^2
		(-{s})^{\frac n2-1} \dd s \leq C\norm{u}^2_{\Ll^p([0,\infty))}\label{eq:laplaceboundspectralmeasure}
	\end{equation}
 for some $C > 0$.
	Using the Hölder inequality, we have
	\begin{equation}
	 \begin{aligned}
		\int_{-\infty}^{-1} \,\abs*{\int_0^\infty \e^{t{s}} u(t) \dd t}^2
		(-{s})^{\frac n2-1} \dd {s} &=
		\int_{1}^{\infty} \,\abs*{\int_0^\infty \e^{-ty} u(t) \dd t}^2
		y^{\frac n2-1} \dd y \\
	&\leq \int_1^\infty \norm{\e^{-(\,\cdot\,) y}}_{\Ll^{p'}([0,\infty))}^2
		\norm{u}_{\Ll^p([0,\infty))}^2 y^{\frac n2 -1} \dd y.\label{eq:hölderspectralmeasure}
	   \end{aligned}
	\end{equation}
	As $\norm{\e^{-(\,\cdot\,) y}}_{\Ll^{p'}([0,\infty))}^2 = (yp')^{-2/{p'}}$, the integral above stays bounded if 
		\begin{equation}
			-\frac{2}{p'}+\frac n2-1 < -1 \Leftrightarrow
            p > \frac{4}{4-n}.\label{eq:conditiononp}
		\end{equation}
	Since
	$\delta_0$ only defines an element of $X_{-1}$ for $n \leq 3$, one arrives
	at the condition 
    $p > 4/3$ in the one-dimensional setting, $p > 2$ in two dimensions
	and $p > 4$ for $n = 3$.
 \end{example}
	\subsection{The Laplacian on a half line}
 \begin{example}
	We consider the Dirichlet controlled heat equation on the half axis $[0,\infty)$.
	Set $X = \Ll^2([0,\infty))$ and consider the Laplacian $A = \frac{\ddd^2}{\ddd t^2}$ on the domain $\mathcal{D}(A) = 
	\Hh^1_0([0,\infty)) \cap \Hh^2([0,\infty))$. We select the scalar input space $U = \C$ and choose the control operator
	$b = \delta_0' \in X_{-1} \simeq \LL(\C,X_{-1})$. Our aim is to investigate $p$\kern+.06em{}-admissibility of $b$ for this example system.

	We want to compute the resolvent of $A$ to aid us in finding the spectral 
	measure of $A$. To that end, for $g \in \Ll^1([0,\infty)) \cap \Ll^2([0,\infty))$, define the odd extension $\tilde g \in \Ll^1(\R) \cap \Ll^2(\R)$ by
		\begin{equation*}
			\tilde g(x) = \begin{cases}
								g(x) & x \geq 0\\
								-g(-x) & x < 0.
							\end{cases}
		\end{equation*}
    Denoting the Fourier variable on the full real axis
	by $\eta$ and employing the representation of the resolvent of the Laplacian by its symbol in the Fourier domain, we have
		\begin{equation*}
			R(z,A) g = \Bigl.\fourier^{-1} \bigl((z+\eta^2)^{-1} (\fourier\tilde g)(\eta)\bigr)\Bigr|_{[0,\infty)},
		\end{equation*}
	which allows for an explicit integral representation of the resolvent. In fact,
 we have
	\begin{align*}
		R(z,A) g(x) &= \frac{1}{\sqrt{2\pi}}\fourier^{-1}\biggl( \frac{1}{z+\eta^2} \int_{-\infty}^\infty \e^{-\I y \eta} \tilde g(y) \dd y \biggr)(x)\\
		&= \frac{1}{\sqrt{2\pi}}\fourier^{-1} \biggl( \frac{1}{z+\eta^2} \int_0^\infty (\e^{-\I y \eta} - \e^{\I y \eta}) g(y) \biggr)(x)\\
		&= \frac{2}{\I\sqrt{2\pi}}\fourier^{-1} \biggl( \frac{1}{z+\eta^2} \int_0^\infty \sin(y\eta) g(y) \dd y \biggr)(x) \\
		&= \frac{1}{\pi \I}\int_{-\infty}^\infty \frac{\e^{\I x \eta}}{z + \eta^2}\int_0^\infty \sin(y\eta) g(y) \dd y \dd \eta \\
		&= \frac{1}{\pi \I}\int_0^\infty \frac{\e^{\I x \eta} - \e^{-\I x \eta}}{z+\eta^2}\int_0^\infty \sin(y\eta) g(y) \dd y \dd \eta \\
		&= \frac{2}{\pi}\int_0^\infty \frac{\sin(x \eta)}{z+\eta^2}\int_0^\infty \sin(y\eta) g(y) \dd y \dd \eta.
	\end{align*}
	For brevity, in the sequel we denote the one-sided sine transform by $\mathcal{S}$, i.e.
	\[(\mathcal{S}f)(x) \coloneq \int_0^\infty f(\xi) \sin(x\xi) \dd \xi,\qquad f \in \Ll^1([0,\infty)) \cap \Ll^2([0,\infty)).\]
	Still considering $g \in \Ll^1([0,\infty)) \cap \Ll^2([0,\infty))$ we compute
	\begin{align*}
		\spr{g}{R(z,A){g}}_{\Ll^2([0,\infty))} &= \int_0^\infty {g}(t) \overline{(R(z,A){g})(t)} \dd t \\
		&= \frac 2\pi \int_0^\infty {b}(t) \int_0^\infty \frac{\sin(\tau t)}{\overline z + \tau^2} \int_0^\infty \sin(\tau s) \overline{{g}(s)}\dd s \dd \tau \dd t\\
		&= \frac 2\pi \int_0^\infty {g}(t) \int_0^\infty \frac{\sin(\tau t)}{\overline z + \tau^2} (\mathcal{S}\,\overline {g})(\tau) \dd \tau \dd t \\
		&= \frac 2\pi \int_0^\infty \frac{(\mathcal{S}\,\overline {g})(\tau)}{\overline z + \tau^2} \int_0^\infty \sin(\tau t) {g}(t) \dd t \dd \tau \\
		&= \frac 2\pi \int_0^\infty \frac{(\mathcal{S}\,\overline {g})(\tau)}{\overline z + \tau^2} (\mathcal{S} {g})(\tau) \dd \tau \\
		&= \frac 2\pi \int_0^\infty \frac{\abs*{(\mathcal{S} {g})(\tau)}^2}{\overline{z} + \tau^2} \dd \tau.
	\end{align*}
	We set $\rho = \tau^2$ and get
	\begin{align*}
		\spr{g}{R(z,A){g}}_{\Ll^2([0,\infty))} &= \frac 1\pi \int_0^\infty \frac{\bigl\lvert(\mathcal{S}{g})(\sqrt{\rho})\bigr\rvert^2}{\overline{z} + \rho} \frac{\ddd\rho}{\sqrt\rho}\\
		&= \frac{1}{\pi} \int_{\C} \chi_{[0,\infty)}(\rho) \frac{{\bigl|(\mathcal{S}{g})(\sqrt{|\rho|})\bigr|}^2}{\overline{z} + \rho} \frac{\ddd\rho}{\sqrt{\lvert\rho\rvert}}\\
		&= \frac{1}{\pi} \int_{\C} \chi_{(-\infty,0]}(\rho) \frac{{\bigl|(\mathcal{S}{g})(\sqrt{|\rho|})\bigr|}^2}{\overline{z} - \rho} \frac{\ddd\rho}{\sqrt{\lvert\rho\rvert}}\\
		&= \int_{\C} \frac{1}{\overline{z} - \rho} \dd E_{g,g}(\rho),
	\end{align*}
	where $E_{g,g}$ is Lebesgue measure weighted with the function 
	$\chi_{(-\infty,0]} (\,\cdot\,){\lvert(\mathcal{S}{g})(\sqrt{\lvert\,\cdot\,\rvert})\rvert^2}/(\pi{\sqrt{\lvert\,\cdot\,\rvert}})
	$. We note that---like the Fourier transform---the sine and cosine transforms extend to operators defined on $\Ll^2$ functions by density; cf.\ e.g.\ \cite{goldbergCertainOperatorsFourier1959}.

    Returning to the control operator $b = \delta_0'$, we note that
    \[(\mathcal{S}\delta'_0)(t) = \langle \delta'_0 , \sin((\,\cdot\,) t) \rangle = \bigl.-\frac{\ddd}{\ddd \xi} \sin(t\xi)\bigr|_{\xi = 0} = -t.\]
    In the same manner as for the Laplacian on the full Euclidean space as in Section~\ref{sec:laplacianfullspace} above,
    we can determine the measure $\mu$ from Proposition~\ref{prop:L-C_normal}. This measure $\mu$ is formally given by the expression \enquote{$\spr{E(\,\cdot\,)\delta_0'}{\delta_0'}_X$}---which is not surprising regarding the extension procedure used by Weiss in \cite{weissPowerfulGeneralizationCarleson1999} mentioned in Remark~\ref{rem:extensionEbb} above---giving a measure that is equivalent 
    (in the sense of mutual absolute continuity) to the measure defined by
     \[E_{\delta'_0,\delta'_0}(z) = \chi_{(-\infty,0]} (z)\frac{\lvert(\mathcal{S}\delta'_0)(\sqrt{\lvert z\rvert})\rvert^2}{\sqrt{\lvert z\rvert}}\lambda^1(z)= \chi_{(-\infty,0]}(z) \abs{z}^{1/2}\lambda^{1}(z);\]
     note also that the associated Radon--Nikodým densities are bounded away from $0$ and $\infty$.
    
    In terms of $p$\kern+.06em{}-admissibility of $b$, the computation in \eqref{eq:laplaceboundspectralmeasure}, \eqref{eq:hölderspectralmeasure} and \eqref{eq:conditiononp} then yields
    finite-time admissibility of $b$ for all $p > 4$.
 \end{example}
 \begin{example}
	Proceeding in a similar manner as in the prior case of Dirichlet control, we can consider the Neumann controlled heat equation
	on the non-negative real axis. On the same state space $X = \Ll^2([0,\infty))$ we now consider the 
	the Laplace operator $A = \frac{\mathrm{d}^2}{\mathrm{d} t^2}$ on the domain $\mathcal{D}(A) = 
	\{ f \in \Hh^2([0,\infty)) : f'(0) = 0 \}$ and the control operator $b = \delta_0$
	on the input space $U = \C$.

	Analogously as in the Dirichlet setting, but using the even extension of functions
	defined on the half line, the resolvent of $A$ computes to 
	\[
		(R(z,A)g)(x) = \frac 2\pi \int_0^\infty \frac{\cos({x\eta})}{z + \eta^2}
		\int_0^\infty \cos(y\eta) g(y) \dd y \dd \eta
	\]
 for $g \in \Ll^1([0,\infty)) \cap \Ll^2([0,\infty))$.
	For the spectral measure, an analogous calculation as in the Dirichlet case yields 
	\[
		E_{g,g}(z) = \chi_{(-\infty,0]}(z)\frac{\lvert(\mathcal{C}x)(\sqrt{\lvert z \rvert})\rvert^2}{\pi\sqrt{\lvert z \rvert}} \lambda^{1}(z),\qquad g \in \Ll^1([0,\infty)) \cap \Ll^2([0,\infty)).
	\]
	Here $\mathcal{C}$ denotes the one-sided cosine transform defined by $(\mathcal{C}f)(x) \coloneq \int_0^\infty f(\xi) \cos(x\xi) \dd \xi$ for $f \in \Ll^1([0,\infty)) \cap \Ll^2([0,\infty))$, and we also denote by $\mathcal{C}$ the extension of the transform to $\Ll^2([0,\infty))$ and to tempered distributions.
	
    As $\mathcal{C} \delta_0 \equiv 1$, the same argument as in the Dirichlet case above lets us determine the measure $\mu$ from Proposition~\ref{prop:L-C_normal} to be equivalent to 
		\[
			E_{\delta_0,\delta_0}(z) = \chi_{(-\infty,0]}(z) \abs{z}^{-1/2}\lambda^{1}(z).
		\]
    Since this is the same density as in the full space case with $n = 1$ in 
    Example~\ref{ex:laplacianfullspace}, the computation in 
    \eqref{eq:laplaceboundspectralmeasure}, \eqref{eq:hölderspectralmeasure} 
    and \eqref{eq:conditiononp} there yields finite-time $\Ll^p$-admissibility 
    of $b = \delta_0$ for $A$ for all $p > \frac 43$.
 \end{example}
	\section{Applications to infinite-dimensional input spaces}
 \label{sec:applicationsinfdim}
     Laplace--Carleson embeddings have been proven to be useful in the context
     of admissibility with finite-dimensional input spaces. In this section we
     propose 
     a partial extension to infinite-dimensional input spaces. We note that the 
     question of a full characterization of $p$\kern+.06em{}-admissibility in terms
     of these criteria---without restrictive additional 
     assumptions---seems to be open at this moment in time.
 
	However, imposing additional assumptions on the system, we can extend
    admissibility results for finite-dimensional input spaces to 
    infinite-dimensional ones. One of such  conditions is that \emph{$A$ has 
    the $p$\kern+.06em{}-\kern-.15em{}Weiss property.}
   
	\begin{definition}[$p$\kern+.06em{}-\kern-.15em{}Weiss property]
	Let $p\in[1,\infty]$ and $A$ generate a $\mathrm{C}_0$-semigroup on the Banach space $X$. Then we say that the \emph{$p$\kern+.06em{}-\kern-.15em{}Weiss property for control operators} holds
	for $A$ if for any Banach space $U$ and any $B\in\LL(U,X_{-1})$, we have that
	\[
	\text{$B$ is infinite-time $p$\kern+.06em{}-admissible }\Leftrightarrow 
	\sup_{\Re \lambda > 0}(\Re \lambda)^{\frac 1p}\norm{R(\lambda,A) B} < \infty.
	\]
    The \emph{$p$\kern+.06em{}-\kern-.15em{}Weiss property for observation operators} holds if for all  $C \in \LL(X_1,Y)$, $Y$ Banach space, 
    \[
	\text{$C$ is infinite-time $p$\kern+.06em{}-admissible }\Leftrightarrow 
	\sup_{\Re \lambda > 0}(\Re \lambda)^{\frac{1}{p'}}\norm{C R(\lambda,A)} < \infty.
	\]
	\end{definition}
 
 For general semigroup generators $A$, the $p$\kern+.06em{}-\kern-.15em{}Weiss property does not hold, as indicated in the introduction; 
    see e.g.\ \cite{jacobADMISSIBLEWEAKLYADMISSIBLE2002,jacobDisproof2000,zwartWeakAdmissibilityDoes2003a} for $p = 2$. In special cases---such as the case of self-adjoint generators $A$---it does hold; see Proposition~\ref{prop:selfadjoint}
    below.
 
    \begin{theorem}[Uniform weak admissibility \& Weiss property]
  \label{prop:uniformweiss}
  Let $A$ be the generator of a $\mathrm{C}_0$-semigroup on
  the Banach space $X$. Fix
  an exponent $q \in (1,\infty)$ and assume that the $q$-\kern-.15em{}Weiss property holds for $A$. Then the following are equivalent for $B \in \LL(U,X_{-1})$\textup{:}
    \begin{enumerate}[leftmargin=3\parindent,rightmargin=3\parindent,label=\textup{(}\kern-.05em{}\alph*\textup{)}]
    \item $B$ is $\Ll^q$-admissible for $A$.
    \item For all
  $u \in U$ with $\norm{u}_U =1$ the input element $Bu \in X_{-1}$ is $\Ll^q$-admissible 
 and the admissibility constants, i.e.\ the numbers \[C(u) \coloneq \inf \biggl\{C>0 : \forall v\in\Ll^q([0,t])\;\norm*{\int_0^t T(t-s) Bu v(s) \dd s} \leq C \norm{v}_{\Ll^q([0,t])} \biggr\},\] are uniformly bounded in $u$.
    \end{enumerate}
    \end{theorem}
    \begin{proof}{ 
    We note that for any $u \in U$ with norm 1
    and any $v \in \Ll^q([0,\infty))$ we have $uv \in \Ll^q([0,\infty),U)$. To show \enquote{$\text{(a)} \Rightarrow \text{(b)}$}, we assume that $B\colon U \to X_{-1}$ is $q$\kern+.06em{}-admissible. Then by admissibility of $B$ we infer
    \[
        \norm*{\int_0^t T(t-s) B(uv)(s) \dd s}_{X} \leq C \norm{uv}_{\Ll^q([0,t], U)} = C
         \norm{v}_{\Ll^q([0,t])},
    \]
    which implies that $Bu \in X_{-1}$ is $q$\kern+.06em{}-admissible with admissibility constant that does not depend on $u$.}
    Hence the admissibility constants are also uniformly bounded in $u$.
    
    For the other implication, we assume that for all $u \in U$ with unit norm the input
    element $Bu$ is $q$\kern+.06em{}-admissible, i.e.\ 
    \[
        \norm*{\int_0^t T(t-s) B(uv)(s) \dd s}_{X} \leq C(u)
         \norm{v}_{\Ll^q([0,t])},
    \]and that the associated constants are uniformly bounded in $u$, which is to say we have \[ \sup_{u\in U,\norm{u}_U = 1} C(u) < \infty.\] 
    
    Hence the mapping $\Phi_{t,Bu}\colon \Ll^q([0,\infty)) \to X, \, v \mapsto \int_0^t T(s) Buv(s) \dd s$
        allows for the bound $M_t \coloneq \sup_{u\in U, \norm{u} = 1} \norm{\Phi_{t,Bu}} < \infty$.
    Using infinite-time $q$\kern+.06em{}-admissibility of $Bu$, we also get a bound of the form  {$M \coloneq \limsup_{t\to\infty} M_t < \infty$}.
    Furthermore, setting $\tilde v(s) \coloneq (q\Re \lambda)^{1/q} \e^{-\lambda s}$ we get $\norm{\tilde v}_{\Ll^q([0,\infty))} = 1$.
    Thus $\norm{\Phi_{t,Bu} \tilde v}_X \leq \norm{\Phi_{t,Bu}}\norm{\tilde v}_{\Ll^q([0,\infty))}  = \norm{\Phi_{t,Bu}} \leq M$
    and
    \[
        \norm*{\int_0^\infty T(s) Bu \tilde v(s) \dd s}_X = \lim_{t\to\infty} \norm{\Phi_{t,Bu} \tilde v}_X \leq M.
    \]
    As the resolvent is the Laplace transform of the semigroup, it follows that
    \begin{align*}
        (\Re \lambda)^{1/q} \norm{R(\lambda, A) Bu}_X &= 
        (\Re \lambda)^{1/q} \norm*{\int_0^\infty \e^{-\lambda s}T(s) Bu\dd s}_X \\
        &= \norm{\Phi_{t,Bu} \tilde v}_X \leq M.
    \end{align*}
    Due to our assumption, $\norm{\Phi_{t,Bu}}$ is uniformly bounded, which implies
    \[
        (\Re \lambda)^{1/q} \norm{R(\lambda, A) B} = 
        \sup_{{u\in U,\norm{u}_U = 1}}(\Re \lambda)^{1/q}\,\norm{R(\lambda, A) Bu}_X \leq M;
    \]
    taking the supremum among $\lambda \in \C$ with $\Re \lambda > 0$ then gives
    \[
       \sup_{\Re\lambda > 0} (\Re \lambda)^{1/q} \,\norm{R(\lambda, A) B}_X \leq M < \infty.
    \]
    Hence $B$ is infinite-time $q$\kern+.06em{}-admissible by the assumed validity of the $q$\kern+.06em{}-\kern-.15em{}Weiss property for $A$.
    \end{proof}
    \begin{remark}
     Theorem \ref{prop:uniformweiss} allows extending the Laplace--Carleson test for admissibility to infinite-dimensional input spaces. The practicality of this condition is yet to be seen, but this approach allows for assessing admissibility sharply in $p$. For example, a possible application to the heat equation could involve estimates
     for the norm of the trace of eigenfunctions on the boundary by means of eigenfunction expansions, as the eigenfunctions of the Laplacian with Dirichlet or Neumann boundary conditions are explicitly known for well-behaved geometries. Moreover, there is a wide array of work on norms of the appropriate traces of eigenfunctions on the boundary, see e.g.\ \cite{grebenkovGeometricalStructureLaplacian2013,zelditchBilliardsBoundaryTraces2003,tataruRegularityBoundaryTraces1998,hassellUpperLowerBounds2002,backerBehaviourBoundaryFunctions2002}.
    \end{remark}
 In this context we note that the $p$\kern+.06em{}-\kern-.15em{}Weiss property is strongly related to 
 admissibility of $(-A)^{1/p}$ for $A$, which was shown by Haak in \cite{haakKontrolltheorieBanachraeumenUnd2005}. If $A$ generates a 
 semigroup of self-adjoint operators and $\sigma(A) \subset (-\infty,0]$, this even leads to
 validity of the $p$\kern+.06em{}-\kern-.15em{}Weiss property for $p \leq 2$.
 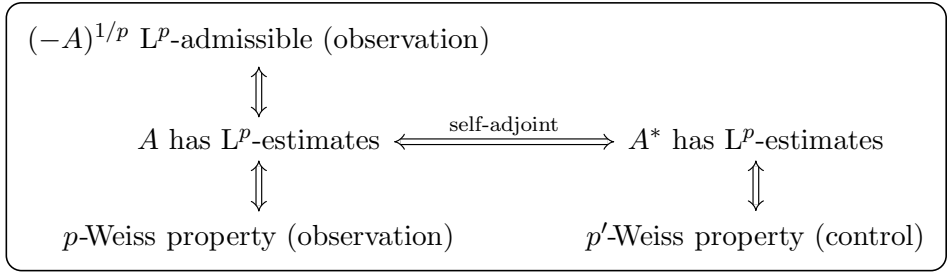
\begin{figure}[ht]
     \centering
        \begin{tikzcd}[every matrix/.append style={draw=black, line width=.6pt, inner sep=3pt, rounded corners=.2cm}]
        	{(-A)^{1/p} \ \mathrm{L}^p\text{-admissible (observation)}} \\
        	{A \text{ has }\mathrm{L}^p \text{-estimates}} & {A^* \text{ has }\mathrm{L}^p \text{-estimates}} \\
        	{p\text{\kern+.06em{}-\kern-.15em{}Weiss property (observation)}} & {p'\text{\kern-.03em{}-\kern-.15em{}Weiss property (control)}}
        	\arrow[Leftrightarrow, "\text{self-adjoint}", from=2-1, to=2-2]
        	\arrow[Leftrightarrow, from=2-2, to=3-2]
        	\arrow[Leftrightarrow, from=2-1, to=3-1]
        	\arrow[Leftrightarrow, from=1-1, to=2-1]
        \end{tikzcd}
     \caption{The $p$\kern+.06em{}-\kern-.15em{}Weiss property for self-adjoint $A$, see \cite{haakKontrolltheorieBanachraeumenUnd2005}.}
     \label{fig:weisspropselfadjoint}
     \vspace{-.5cm}
\end{figure}
  \begin{proposition}
     \label{prop:selfadjoint}
     Let $A$ be a self-adjoint operator on a Hilbert space $X$ with spectrum
     contained in $(-\infty,a]$ for some $a\in\R$. Then $A$ has the $p$\kern+.06em{}-\kern-.15em{}Weiss
     property \textup{(}\kern+.05em{}for control operators\textup{)} if $p \leq 2$.
 \end{proposition}
 \begin{proof}
    The case $p=2$ is contained in \cite[Corollary~4.3]{lemerdyWeissConjectureBounded2003} as the negative semidefinite self-adjoint operator $A$ is the generator of a $\mathrm{C}_0$-semigroup
    of contractions on $X$.
    
    Thus we fix $p < 2$. The claim is that $B \in \LL(U,X_{-1})$ is infinite-time $\Ll^p$-admissible if and only if the quantity $\sup_{\lambda > 0} \lambda^{1/p} \norm{R(\lambda,A) B}_{\LL(U,X)}$
    is finite. To prove the claim we use results from Hosfeld, Jacob and Schwenninger \cite{hosfeldCharacterizationOrliczAdmissibility2023} in conjunction with results from Haak \cite{haakKontrolltheorieBanachraeumenUnd2005}.
    
    First we note that due to the assumption of being bounded from above, the self-adjoint $A$ generates a bounded analytic semigroup $(T(t))_{t\ge0}$, see \cite[Corollary~II.4.7]{engel_one-parameter_2000}. 
We observe that the spectral theorem for unbounded self-adjoint operators allows us to prove 
    $p'$\kern-.06em{}-admissibility of $(-A)^{1/p'}$ as an observation operator for $p' > 2$ more directly. Indeed, as $-A$ is self-adjoint, 
    there exists a spectral measure $E$ such that 
\[
    -A = \int_{\sigma(-A)} z \dd E(z),\qquad T(t) = \int_{\sigma(-A)} \e^{-tz} \dd E(z), \quad t\ge0,
\]
holds. In particular, there is a representation of the 
norm, i.e.
\begin{equation*}
    \int_0^\infty \norm{(-A)^{1/{p'}} T(t) x}_X^{p'} \dd t
    =
    \int_0^\infty \biggl( \int_{\sigma(-A)} \abs{z^{1/{p'}} \e^{-tz}}^2 \dd E_{x,x}(z) \biggr)^{{p'}/2} \dd t {;}\label{eq:normrepresentation}
\end{equation*}
 {cf.\ \cite[Theorem~4.7]{conwayCourseFunctionalAnalysis2007}.}
For $\sigma$-finite measure spaces $(X,\mu)$ and $(Y,\nu)$, a measurable function $f$ defined on $X \times Y$ and $1 \leq p < q < \infty$, Minkowski's integral inequality reads
\[
\int_Y \biggl( \int_X \abs{f(x,y)}^p \dd \mu(x) \biggr)^{q/p} \dd \nu(y) \leq
\biggl( \int_X \biggl( \int_Y \abs{f(x,y)}^q \dd \mu(y)\biggr)^{p/q} \dd \mu(x) \biggr)^{q/p},
\]
see e.g.\ \cite[Exercise~3.10.46]{bogachevMeasureTheory2007} or \cite[Proposition~1.2.22]{hytonenAnalysisBanachSpaces2016}.
Applying Minkowski's inequality in \eqref{eq:normrepresentation}
above yields
\begin{align*}
    \int_0^\infty \biggl(\int_{\sigma(-A)} \abs{z^{1/{p'}} \e^{-tz}}^2 \dd E_{x,x}(z) \biggr)^{{p'}\!/2} \dd t &\leq \biggl( 
    \int_{\sigma(-A)} \biggl( \int_0^\infty z \e^{-tz{p'}} \dd t\biggr)^{2/{p'}} \dd 
    E_{x,x}(z)\biggr)^{{p'}\!/2}\\
    &= \Bigl(\frac{1}{{p'}^{2/{p'}}}
    \int_{\sigma(-A)} \dd E_{x,x}(z)\Bigr)^{{p'}\!/2}\\
    &= \frac{1}{p'} \norm{x}_X^{p'}.
\end{align*}
Therefore $(-A)^{1/p'}$ is $p'$\kern-.06em{}-admissible as an observation operator, which is equivalent to requiring that the operator $-A$ satisfies $\Ll^{p'}$\kern-.17em{} estimates in the sense of Haak, \cite[page~29]{haakKontrolltheorieBanachraeumenUnd2005}.

    Because $A$ was required to be self-adjoint,
    the same estimates hold for $(-A)^*$.
    Hence we can apply the generalization of \cite[Satz~2.1.6]{haakKontrolltheorieBanachraeumenUnd2005} given in \cite[page~29]{haakKontrolltheorieBanachraeumenUnd2005}, which states that the 
    $p$\kern+.06em{}-\kern-.15em{}Weiss property for control operators is valid for $A$. This ends the proof.
    \end{proof}
 \begin{remark}
     The last proposition leads us to the observation that there is a noticeable
     difference between the cases $p < 2$ and $p > 2$ in checking $p$\kern+.06em{}-admissibility.
     In the control operator setting with a system on a Hilbert space $X$, the case $p < 2$ is simpler than $p > 2$. By the previous argument, once we transfer the reasoning from observation operators to control operators---$A$ was self-adjoint---we would need to require $p'\!/2 \geq 1$ for the Minkowski inequality to hold, which means that we would need
     $p < 2$.

     The complementary case does not allow for such arguments. Below we give an example of a self-adjoint and negative operator that does not have the $4$-\kern-.15em{}Weiss property.  We note in passing that in the case $p = \infty$, the property of $A^{*}$ being an $\Ll^{1}$-admissible observation operator is equivalent to $A$ being bounded on $X$, see also \cite{jacobRefinementBaillonTheorem2022}.
 \end{remark}
 \begin{example}
     By means of the example of the one-dimensional heat equation from Section~\ref{sec:dirichletheateq} above, we can show that a negative definite self-adjoint generator need not possess the $p$\kern+.06em{}-\kern-.15em{}Weiss property
     for $p > 2$.
     
     As a continuation of our earlier computations, we prove that the resolvent condition is valid for $p = 4$. Since the control operator of system in question was not $4$-admissible, this implies that the generator does not have the $4$-\kern-.15em{}Weiss property.

     Recall that the Dirichlet Laplacian $A$ on $\Ll^2([0,1])$ has the diagonal representation
     \[
        Ae_n = - n^2 \pi^2 e_n, \qquad n\in \N,
     \]
     where $e_n$ are the unit vectors in $\ell^2(\N)$. Moreover, the control operator $b$ modeling the heat injection at the right endpoint of the interval was described by the sequence $b = \{ b_n \}_{n\in\N}$, where $b_n = (-1)^n \sqrt{2}n\pi$.

     The composition $R(\lambda,A)b$ of the resolvent and control operator is then 
     given by the sequence
     \[
        (R(\lambda,A) b)_n = \frac{(-1)^n \sqrt{2} n \pi}{\lambda + n^2\pi^2},\qquad n\in\N.
     \]
     Validity of the  resolvent condition in the $4$-\kern-.15em{}Weiss property is (by definition) equivalent to uniform boundedness
     of the quantity $(\Re \lambda)^{1/4} \norm{R(\lambda,A)b}_{\LL(U,X)}$ among complex numbers $\lambda$ with positive real part.
     
     In view of $U = \C$ and $X = \ell^2(\N)$,  we insert the expressions for the coefficients of $R(\lambda,A)b$; it then remains to show that
     \begin{equation}
        \sup_{\Re \lambda > 0} (\Re \lambda)^{1/4} \Biggl(\sum_{n=1}^\infty \frac{2 n^2 \pi^2}{\abs{\lambda + n^2\pi^2}^2} \Biggr)^{1/2} < \infty.\label{eq:4Weiss}
     \end{equation}
     Note that we can restrict to real $\lambda$. Moreover, it is convenient to consider $\mu = \sqrt{\lambda}$ in place of $\lambda$. Employing this substitution, after squaring the expression \eqref{eq:4Weiss} becomes
     \begin{equation*}
        \sup_{\mu > 0} \mu \sum_{n=1}^\infty \frac{2 n^2 \pi^2}{(\mu^2 + n^2\pi^2)^2} < \infty.
     \end{equation*}
     A closed form of the sum $\sum_{n=1}^\infty \frac{\pi^2 n^2}{(b +  \pi^2 n^2)^2}$
	exists in terms of hyperbolic functions. 
	Indeed, using  {\[\coth x \coloneq \frac{\cosh x}{\sinh x} = \frac{\e^x + \e^{-x}}{\e^x - \e^{-x}},\quad \csc x \coloneq \frac{1}{\sin x},\quad  \csch x \coloneq \frac{1}{\sinh x} = \frac{2}{\e^x - \e^{-x}}\] together with the series expansions}
    \begin{equation*}
			\coth(\pi x) = \frac{1}{\pi x} + \frac{2 x}{\pi} \sum_{k = 1}^\infty
			\frac{1}{k^2 + x^2},\qquad
			\csc^2(\pi x) = \frac{1}{\pi^2} 
			\sum_{k=-\infty}^\infty \frac{1}{(x - k)^2},
		\end{equation*}
	see \cite[page 36]{gradshtein_table_1980}, and the identity $\csch^2 (x)= -\kern-.1em{}\csc^2(\I x)$ for $x \in \R$, one can show that
     \[
        \mu \sum_{n=1}^\infty \frac{2 n^2 \pi^2}{(\mu^2 + n^2\pi^2)^2} = \frac 12 (\coth \mu - \mu \csch^2 \mu) = 
        \frac 12 + \frac{(2 - 4\mu) \e^{2\mu} - 2}{2(\e^{2\mu} - 1)^2}.
     \]
     As $\frac{(2 - 4\mu) \e^{2\mu} - 2}{2(\e^{2\mu} - 1)^2} < 0$, we have
     \begin{equation*}
        \sup_{\mu > 0} \mu \sum_{n=1}^\infty \frac{2 n^2 \pi^2}{(\mu^2 + n^2\pi^2)^2} = \sup_{\mu > 0} \biggl(\frac 12 + \frac{(2 - 4\mu) \e^{2\mu} - 2}{2(\e^{2\mu} - 1)^2}\biggr)
        \leq \frac 12;
     \end{equation*}
     validating the resolvent condition. 
     Therefore the $4$-\kern-.15em{}Weiss property is violated for $A$.
 \end{example}
 The above example shows that Theorem~\ref{prop:uniformweiss} cannot be applied to show $p$\kern+.06em{}-admissibility for systems arising from boundary control with infinite-dimensional input space if $p>2$, even if $A$ is self-adjoint. Consequently, determining admissibility with respect to sharp values of $p$ remains difficult in this regime. Practically, the straight-forward tool to derive admissibility seems to be trace estimates in conjunction with the interpolation spaces as lined out in Section~\ref{sec:interpolation}.
\emergencystretch=1em
\printbibliography
\end{document}